\documentclass[hidelinks,onefignum,onetabnum]{siamart250211}

\usepackage{lipsum}
\usepackage{amsfonts}
\usepackage{graphicx}
\usepackage{epstopdf}
\usepackage{algorithmic}
\ifpdf
  \DeclareGraphicsExtensions{.eps,.pdf,.png,.jpg}
\else
  \DeclareGraphicsExtensions{.eps}
\fi

\newsiamremark{remark}{Remark}
\newsiamremark{hypothesis}{Hypothesis}
\crefname{hypothesis}{Hypothesis}{Hypotheses}
\newsiamthm{claim}{Claim}
\newsiamremark{fact}{Fact}
\crefname{fact}{Fact}{Facts}

\headers{Time-varying Job-swtiching costs}{G. Ha, J. Jeon, and J. OK}

\title{Finite-Horizon Optimal Consumption and Investment with Time-Varying Job-Switching Costs
}

\author{Gugyum Ha\thanks{Department of Mathematics, Sogang University, Seoul 04107, Republic of Korea 
  (\email{ggha@sogang.ac.kr}, \email{jihoonok@sogang.ac.kr}).}
\and Junkee Jeon\thanks{Department of  Applied Mathematics, Kyung Hee University, Yongin-Si 17104, Republic of Korea
  (\email{junkeejeon@khu.ac.kr}).}
\and Jihoon Ok\footnotemark[1]}

\usepackage{amsopn}

\ifpdf
\hypersetup{
  pdftitle={Optimal Consumption, Investment, and Job-Switching with Time-Varying Switching Costs},
  pdfauthor={G. Ha, J. Jeon, and J. OK}
}
\fi

\usepackage{mathrsfs}
\usepackage{amsfonts}
\usepackage{enumitem}
\usepackage{amssymb}
\usepackage{amscd}
\usepackage{epsfig}
\usepackage{graphics}
\usepackage{multirow}
\usepackage{amsmath}

\newtheorem{pr}{Problem}[section]

\newtheorem{thm}{Theorem}[section]
\newtheorem{lem}[thm]{Lemma}

\newtheorem{as}{Assumption}[section]
\newtheorem{rem}{Remark}[section]

\def\a{\alpha}
\def\b{\beta}
\def\d{\delta}
\def\D{\Delta}
\def\e{\varepsilon}
\def\g{\gamma}
\def\l{\lambda}

\def\o{\omega}
\def\p{\psi}
\def\pat{\partial}
\def\r{\rho}
\def\s{\sigma}
\def\t{\theta}

\def\z{\zeta}
\def\k{\kappa}
\def\gam{\frac{\gamma_1-\gamma}{\gamma_1}}

\begin{document}

\maketitle

\begin{abstract}
In this paper, we study the finite-horizon problem of an economic agent’s optimal consumption, investment, and job-switching decisions.  
The key new feature of our model is that the job-switching cost is time-varying.  
This extension leads to a novel mathematical characterization: the agent’s dual problem reduces to a parabolic double obstacle problem with time-dependent upper and lower obstacles.  
By employing rigorous PDE theory, we establish not only the existence and uniqueness of the solution to this double obstacle problem, but also the smoothness of the two free boundaries that emerge from it.  
Building on these results, we characterize the agent’s optimal consumption, portfolio, and job-switching strategies.
\end{abstract}



\section{Introduction}
With the growing diversity of job types and work arrangements, labor flexibility has become increasingly important, leading to a surge of studies on job-switching behavior within the utility maximization framework (see, e.g., \cite{an2025,JP23,Lee-et-al-2019,Qi25,Qi26,ShimJeon2025,SKS,SS14}, and references therein). However, most of these studies consider infinite-horizon problems without a mandatory retirement date, which makes them less applicable to real-world situations where retirement is inevitable.

To address this limitation, \cite{ZJ2025} investigated a finite-horizon model that incorporates both a mandatory retirement date and job-switching costs. Nevertheless, their analysis assumes that the switching cost is a constant, independent of time. Motivated by the desire to capture more realistic features of the labor market—where switching costs can vary with age, tenure, or macroeconomic conditions—we extend their framework by allowing the job-switching cost to be a time-dependent function.

The overall structure of our paper follows the same spirit as \cite{ZJ2025}. Similar to their study, we consider an agent who can choose between two types of jobs (or job categories). The job offering a higher income provides less leisure compared to the alternative with lower income. The agent faces a fixed mandatory retirement time, before which she can freely switch between the two jobs. However, each job switch incurs a time-dependent switching cost, reflecting the fact that the difficulty or burden of switching may vary over time. Following \cite{ZJ2025}, we employ the dual-martingale approach to transform the agent’s utility maximization problem into a finite-horizon pure optimal switching problem that involves only job-switching decisions. From this formulation, we further reduce the problem to a parabolic double obstacle problem, whose two free boundaries determine the agent’s optimal job-switching decisions.

However, the double obstacle problem that arises in our framework differs substantially from that in \cite{ZJ2025}. Because we model the job-switching cost as a time-varying function, both the upper and lower obstacles of our parabolic double obstacle problem also depend explicitly on time. This additional time dependence introduces significant analytical challenges in establishing the existence, regularity, and smoothness of the solution and its associated free boundaries.

Indeed, to ensure that job switching occurs within a reasonably confined region, we impose additional assumptions on the switching costs. Unlike the conditions in \cite{ZJ2025}, our framework requires restrictions involving the time derivatives of the switching cost functions. Moreover, when deriving the estimates, 
essential for establishing the free boundary regularity as in \cite{ZJ2025}, extra terms naturally emerge, preventing us from obtaining the desired estimate in a clean form. To overcome this difficulty, we partition the domain into two subregions and interpret the double obstacle problem as the conjunction of two single-obstacle problems. Each single-obstacle problem is then approximated by penalized problem, allowing us to derive the desired estimates separately.

The parabolic double obstacle problem has been extensively studied in the mathematical finance literature (see, e.g., \cite{CHEN2012928,DY09,Dai2010,Yi}). In these earlier works, both obstacles possess favorable structural properties, which make it relatively straightforward to determine the sign of the time derivative of the solution. This, in turn, facilitates the analysis of the monotonicity and smoothness of the two associated free boundaries.

However, in more recent studies, the parabolic double obstacle problems arising in financial applications do not allow the sign of the solution’s time derivative to be easily determined, which prevents one from guaranteeing the monotonicity of the free boundaries and makes the proof of smoothness highly challenging (see \cite{Han2024,ZJ2025}). To overcome this difficulty, \cite{Han2024} characterized the time points at which the monotonicity of the free boundaries changes and established smoothness on intervals where monotonicity holds. Meanwhile, \cite{ZJ2025} introduced a domain transformation technique to determine the sign of the time derivative in the transformed domain, thereby deriving the smoothness of the free boundaries.

Nevertheless, in both studies, the two obstacles are constant. In contrast, we extend the method of \cite{ZJ2025} to the case of time-varying obstacles and establish the continuity of both free boundaries under this more general and realistic setting. Moreover, instead of using the result of \cite{Friedman}, we establish the smoothness result of the free boundary as follows. We take into account a sequence of $C^1$ level curves which converges to each free boundaries, as introduced in \cite{DA}. Unlike \cite{DA}, which guaranteed the convergence of the level curves by stochastic methods, we utilize an estimate, which ensures that the level curves are equicontinuous. This implies the local Lipschitz regularity of the free boundary. Applying the boundary Harnack inequality proposed by \cite{Kukuljan2022} and \cite{TorresLatorre2024} further, the smoothness of
the free boundary can also be obtained.

The remainder of this paper is organized as follows. Section~\ref{sec:model} describes the job-switching model with time-varying switching costs. In Section~\ref{sec:OSP}, we apply the dual-martingale approach to derive a finite-horizon optimal switching problem from the agent’s utility maximization problem and obtain the associated parabolic double obstacle problem. Section~\ref{sec:double} establishes the existence and uniqueness of a strong solution to this problem. In Section 5, we analyze the analytical properties and smoothness of the two free boundaries. Finally, Section~\ref{sec:strategy} provides a brief characterization of the optimal strategies.

\section{Model}\label{sec:model}

We consider the optimal consumption, investment, and job-switching problem of an economic agent who works until the mandatory retirement time $T > 0$ in a continuous-time, complete financial market.

The financial market consists of two assets: a risk-free asset $P_0$ and a risky asset (or stock) $P_1$. Their price dynamics are given by $ {dP_{0,t}} = r{P_{0,t}}dt$ and $ {dP_{1,t}} = \mu {P_{1,t}}dt + \sigma {P_{1,t}}dB_t,$
where $r > 0$ is the risk-free interest rate, $\mu \neq r$ and $\sigma > 0$ are the expected return and volatility of the risky asset $P_1$, respectively, and $(B_t)_{t \ge 0}$ is a standard Brownian motion defined on a filtered probability space $(\Omega, \mathcal{F}, \mathbb{F}, \mathbb{P})$. The $\mathbb{P}$-augmentation of the filtration generated by the standard Brownian motion $(B_t)_{t\ge 0}$ is denoted by $\mathbb{F}:=(\mathcal{F}_t)_{t\ge 0}$.  

The agent can choose between two jobs (or job categories), denoted by $\z_0$ and $\z_1$. In job $\z_i$ $(i=0,1)$, the agent receives an income of $\e_i$ and enjoys leisure time $L_i$, where total available leisure time is $\bar{L}$. We assume that job $\z_0$ provides higher income than job $\z_1$, but requires more labor effort, i.e., $0 \le \e_1 < \e_0$ and $0 < L_0 < L_1 < \bar{L}.$

Let $\xi_t$ be an $\mathbb{F}$-adapted, finite variation process over $[0,T]$ that is left-continuous with right limits (LCRL) and takes values in $\{\z_0, \z_1\}$, representing the agent's job choice at time $t$.

The agent can switch freely between the two jobs; however, switching from job $\z_i$ to job $\z_{1-i}$ at time $t$ incurs a fixed, time-dependent, and positive cost $\phi_i(t)$, which is paid from the agent's wealth at the time of switching.

We define the job-indicator process $\eta_t := \mathbf{1}_{\{\xi_t = \z_1\}}$, and if $\pi_t$ denotes the dollar amount invested in the risky asset $P_{1}$ at time $t$, then the corresponding wealth process $W_t^{c,\pi,\xi}$ under the strategy $(c_t, \pi_t, \xi_t)$ evolves as follows: for $t\ge 0$
\begin{align}
       dW_t^{c,\pi,\xi} &= \Big(r W_t^{c,\pi,\xi} + (\mu - r)\pi_t - c_t + (\e_0 \mathbf{1}_{\{\eta_t=0\}} + \e_1 \mathbf{1}_{\{\eta_t=1\}}){\bf 1}_{\{t< T\}}\Big) dt + \sigma \pi_t dB_t \nonumber \\
        &\quad - \Big(\phi_0(t) (\Delta \eta_t)^+ + \phi_1(t) (\Delta \eta_t)^- \Big){\bf 1}_{\{t < T\}}, \quad \text{with} \quad W_0^{c,\pi,\xi} = w,
\end{align}
where $w$ is the initial wealth and $\Delta \eta_t = \eta_{t+} - \eta_t$.

The agent derives utility from consumption $c_t$ and leisure $l_t$, which is modeled using a Cobb-Douglas utility function:
\begin{equation}
    u(c_t, l_t) = \frac{1}{\alpha} \frac{(c_t^{\alpha} l_t^{1-\alpha})^{1-\g}}{1-\g}= \frac{c_t^{1 - \g_1} l_t^{\g_1 - \g}}{1 - \g_1}, \quad \g > 0, \g \neq 1,\text{ and } \g_1:=1 - \alpha (1 - \g),
\end{equation}
where $\alpha$ represents the weight on consumption, and $\g$ is the coefficient of constant relative risk aversion (CRRA).


After mandatory retirement at $T > 0$, the agent ceases working and enjoys his/her maximum leisure time $\bar{L}$, so $l_t$ is given by:
\begin{equation*}
    l_t=L_0 {\bf 1}_{\{\xi_t=\z_0,\;t\in[0,T)\}} +L_1 {\bf 1}_{\{\xi_t=\z_1,\;t\in[0,T)\}} + \bar{L}{\bf 1}_{\{t\in[T,\infty)\}}.
\end{equation*}

The agent's objective is to find the optimal consumption strategy $c_t$, investment strategy $\pi_t$, and job-switching strategy $\xi_t$ that maximize the expected utility given initial wealth $w$ and initial job $j \in \{0,1\}$:
\begin{equation}
   E(j,w; c,\pi,\xi) := \mathbb{E}\bigg[\int_0^T e^{-\b t} u(c_t,L_{\xi_t}) dt + \int_T^\infty e^{-\b t} u(c_t, \bar{L}) dt\bigg],
\end{equation}
where $\b > 0$ is the agent's subjective discount rate.

To ensure the well-definedness of the problem and the existence of a mathematically tractable solution, we impose the following regularity conditions on the job-switching cost functions $\phi_i$:
\begin{as}
For each $i = 0, 1$, the job-switching cost function $\phi_i(\cdot)$ is a positive, bounded, and continuously differentiable function on $[0, T]$.
Specifically, there exists a constant $q>0$ such that $\Vert \phi_i \Vert_{C^{1,1}([0,T])}\le q$.
The positivity of $\phi_i(t)$ ensures that switching jobs always incurs a cost. Boundedness prevents the cost from becoming unrealistically large, and smoothness (differentiability) facilitates mathematical analysis and ensures well-behaved solutions.
\end{as}

Moreover, we focus on a regular situation in which the two free boundaries
are strictly separated and both are nonempty.
For this, we assume the following.
\begin{as}\label{as2}
The switching cost functions $\phi_0(\cdot)$ and $\phi_1(\cdot)$ satisfy the following conditions:
\begin{itemize}
    \item[(i)] $\partial_t \phi_0(t) - r\phi_0(t) < -\partial_t \phi_1(t) + r\phi_1(t)$ for all $t \in [0,T]$.
    \item[(ii)] $\phi_1(0) < (\e_0 - \e_1)\frac{1 - e^{-rT}}{r}$ and $-\partial_t \phi_1(t) + r\phi_1(t) < \e_0 - \e_1$ for all $t \in [0,T]$.
\end{itemize}
\end{as}
Condition(i) will be used later in the Lemma~\ref{reg:s_1} to guarantee that the two free boundaries remain strictly separated.
Condition(ii) implies that the discounted net gain 
$e^{-rt}\big(\mathscr{E}(t)-\phi_1(t)\big)$
is strictly decreasing in $t$, 
where 
\[
\mathscr{E}(t):= (\e_0 - \e_1)\frac{1 - e^{-r(T - t)}}{r}
= \int_t^T e^{-r(s-t)}(\e_0-\e_1)\,ds
\]
denotes the total future income gain from $t$ to $T$ measured in present value at time $t$, which is obtained by switching to the high-income job $\z_0$, at time $t$.
This monotonicity assumption provides an environment for an individual to switch job from $\z_1$ to $\z_0$ (see Remark~\ref{E}).
Since $\phi_1(0)<\mathscr{E}(0)$ and $\mathscr{E}(T)=0$, 
there exists a unique $t_1 \in (0,T)$ such that 
\begin{equation}\label{eq:t1}
\begin{cases}
\phi_1(t) < \mathscr{E}(t), & t\in [0, t_1), \\
\phi_1(t_1) = \mathscr{E}(t_1), &\\
\phi_1(t) > \mathscr{E}(t), & t\in (t_1,T].
\end{cases}
\end{equation}


To describe the initial wealth, note that the wealth level is naturally bounded below by the negative of the present value of future income.
For an agent starting in the high-income job $\z_0$ can borrow up to the present value of future income, $\frac{\e_0(1 - e^{-rT})}{r}$. For an agent starting in $\z_1$, Assumption~\ref{as2} ensures that once wealth reaches $\frac{\e_1(1 - e^{-rT})}{r}$, switching to $\z_0$ is optimal. Consequently, the borrowing limit in this case is $\frac{\e_1(1 - e^{-rT})}{r} - \phi_1(t)$, reflecting both the present value of income and the switching cost.
\begin{as}
    For given initial job $\z_j$, the initial wealth $w$ satisfies 
    $$
    w>\iota(j):=-\frac{\e_0(1-e^{-rT})}{r}(1-j)+\bigg(-\frac{\e_1(1-e^{-rT})}{r}+\phi_1(0) \bigg)j
    $$
\end{as}

Moreover, to ensure that the agent’s optimization problem after the mandatory retirement time is well-defined (see \cite{M69,M71}), we impose the following condition:
\begin{as}\label{as4}
The Merton constant 
\(K_1 := r + \frac{\b - r}{\g_1} + \frac{\g_1 - 1}{\g_1^2}\frac{\t^2}{2}\)
is positive.
\end{as}
Now, we can define the agent's utility maximization problem as follows.
\begin{pr}\label{pr:main}
    Let $j\in\{0,1\}$ and $w>\iota(j)$ be given.
    Then, the agent faces the following maximization problem:
    \begin{equation}
        V(j,w):= \sup_{(c,\pi,\xi)\in{\cal A}(j,w)}E(j,w;c,\pi,\xi),
    \end{equation}
    where ${\cal A}(j,w)$ is the set of all admissible strategies $(c,\pi,\xi)$ satisfying (i) $c_t\ge 0$ and $\pi_t$ are $\mathbb{F}$-progressively measurable and satisfy $\int_0^T (c_t +\pi_t^2)dt<\infty$ a.s., (ii) $(\xi_t)_{t\in[0,T]}\in {\cal O}_j,$ where ${\cal O}_j$ is the collection of all $\mathbb{F}$-adapted, finite variation process over $[0,T]$ that is left-continuous with right limits (LCRL) starting at $j$ and takes values in $\{\z_0, \z_1\},$ (iii) the wealth $W_t^{c,\pi,\xi}$ corresponding to $(c,\pi,\xi)$ satisfies the borrowing constraint
    $ W_t^{c,\pi,\xi}\ge \{ -\frac{\e_0(1-e^{-r(T-t)})}{r} \}{\bf 1}_{\{\xi_t=\z_0\}}+ \{ (-\frac{\e_0(1-e^{-r(T-t)})}{r}+\phi_1(t) ){\bf 1}_{\{t\in[0,T-t_1)\}}+(-\frac{\e_1(1-e^{-r(T-t)})}{r}){\bf 1}_{\{t\in[T-t_1,T]\}}\}{\bf 1}_{\{\xi_t=\z_1\}}$ a.e..
\end{pr}

\section{Derivation of Finite-Horizon Optimal Switching Problem}\label{sec:OSP}

As in \cite[Proposition 3.1]{ZJ2025},  any admissible strategy $(c,\pi,\xi)\in{\cal A}(j,w)$ satisfies the budget constraint 
\begin{equation}\label{eq:budget}
   \begin{aligned}
        \mathbb{E}\bigg[&\int_0^T \Upsilon_t \Big(c_t-\e_0 (1-\eta_t)-\e_1\eta_t\Big)dt +\int_T^\infty \Upsilon_t c_t dt \\&+\sum_{0\le t \le T}\Big(\phi_0(t)\Upsilon_t\big(\D \eta_t)^+ + \phi_1(t)\Upsilon_t\big(\D \eta_t)^-\Big)\bigg]\le w,
   \end{aligned}
\end{equation}
where $\Upsilon_t$ is the stochastic discount factor(SDF) given by 
\begin{equation}
    \label{eq:SDF}
    \Upsilon_t:=e^{-(r+\frac{1}{2}\t^2)t-\t B_t}\qquad\text{with}\qquad\t:=\frac{\mu-r}{\s}.
\end{equation}

From the budget constraint \eqref{eq:budget}, we set the Lagrangian $\mathscr{L}$ as follows: for the Lagrangian multiplier $y>0$ for the budget \eqref{eq:budget}
\begin{small}
    \begin{align}
        \label{eq:Lagrangian}
        \mathscr{L}:=&\mathbb{E}\bigg[\int_0^T e^{-\b t}\Big(u(c_t,L_0)(1-\eta_t) + u(c_t,L_1)\eta_t \Big)dt + \int_T^\infty e^{-\b t} u(c_t, \bar{L}) dt\bigg]\nonumber\\
        &+y\bigg(w-\mathbb{E}\bigg[\int_0^T \Upsilon_t \Big(c_t-\e_0 (1-\eta_t)-\e_1\eta_t\Big)dt +\int_T^\infty \Upsilon_t c_t dt \\
        &\quad\qquad\qquad+\sum_{0\le t \le T}\Big(\phi_0(t)\Upsilon_t\big(\D \eta_t)^+ + \phi_1(t)\Upsilon_t\big(\D \eta_t)^-\Big)\bigg]\bigg)\nonumber\\
        \le &\mathbb{E}\bigg[\int_0^T e^{-\b t}\Big(\big(\widetilde{u}_0(Y_t^y)+\e_0 Y_t^y\big)(1-\eta_t) + \big(\widetilde{u}_1(Y_t^y)+\e_1 Y_t^y\big)\eta_t \Big)dt+\int_T^\infty e^{-\b t}\widetilde{u}_R(Y_t^y)dt\nonumber\\
        &-\sum_{0\le t \le T}e^{-\b t}Y_t^y\Big(\phi_0(t)\big(\D \eta_t)^+ + \phi_1(t)\big(\D \eta_t)^-\Big)\bigg]+yw, \nonumber
    \end{align}
\end{small}
where $Y_t^y:=ye^{\b t} \Upsilon_t,$ and 
\begin{equation}
    \begin{aligned}
        \widetilde{u}_j(y)=&\sup_{c\ge 0}\Big(u(c,L_j)-yc\Big)=\dfrac{\g_1}{1-\g_1}L_j^{\frac{\g_1-\g}{\g_1}}y^{-\frac{1-\g_1}{\g_1}}\quad\text{for } j=0,1,\\
        \widetilde{u}_R(y)=&\sup_{c\ge 0}\Big(u(c, \bar{L})-yc\Big)=\dfrac{\g_1}{1-\g_1}\bar{L}^{\frac{\g_1-\g}{\g_1}}y^{-\frac{1-\g_1}{\g_1}}.
    \end{aligned}
\end{equation}
By direct computation, it is easy to see that for all $t\ge0$
\begin{small}
$$
J_R(y):=\mathbb{E}\bigg[\int_t^\infty e^{-\b(s-t)}\widetilde{u}_R(Y_s^{t,y})ds \mid {\cal F}_t\bigg]=\frac{1}{K_1}\frac{\g_1}{1-\g_1}y^{-\frac{1-\g_1}{\g_1}}\quad\text{with}\quad Y_s^{t,y}:=ye^{\b(s-t)}\frac{\Upsilon_s}{\Upsilon_t}.
$$
\end{small}
From the Lagrangian~\eqref{eq:Lagrangian}, the dual problem is defined as follows:

\begin{pr}\label{pr:dual}
For a given $y>0$ and $j\in\{0,1\}$, we consider the following optimal switching problem:
\begin{align}
\label{dual:J}
J(j,y)
:=&\;
\sup_{(\xi_t)\in{\cal O}_j}
\Bigg\{
\mathbb{E}\Bigg[
\int_0^T e^{-\b t}
\Big(
\big(\widetilde{u}_0(Y_t^y)+\e_0 Y_t^y\big)(1-\eta_t)
+ \big(\widetilde{u}_1(Y_t^y)+\e_1 Y_t^y\big)\eta_t
\Big) dt
\nonumber\\
&\qquad
+ e^{-\b T} J_R(Y_T^y)
- \sum_{0\le t \le T} e^{-\b t} Y_t^y
\Big(
\phi_0(t)(\Delta \eta_t)^+
+ \phi_1(t)(\Delta \eta_t)^-
\Big)
\Bigg]
\Bigg\}.
\end{align}
We refer to $J(j,y)$ as the \emph{dual value function}.
\end{pr}

By a theory of optimal switching control (see \cite[Theorem 5.6]{TangYong1993}), we derive the following system of variational inequalities (VIs) for $P_0$ and $P_1$: on the domain $\Omega_T:=\{(t,y)\;:\; 0< t<T,\; 0<y<\infty\}$
\begin{align}
    \label{eq:vIs}
    \begin{cases}
        \max\{\pat_t P_0+\mathfrak{L} P_0+\widetilde{u}_0+\e_0 y, P_1-P_0 -\phi_0  y \}=0,\\
        \max\{\pat_t P_1+\mathfrak{L} P_1+\widetilde{u}_1 + \e_1 y, P_0-P_1 -\phi_1  y \}=0,\\
        P_0(T,y)=P_1(T,y)=J_R(y)\qquad\text{for all }y\in(0,\infty),
    \end{cases}
\end{align}
where the differential operator $\mathfrak{L}$ is given by $\mathfrak{L}:=\frac{\t^2}{2}y^2\pat_{yy}+(\b-r)y\pat_y -\b.$

Our goal is to find the solution pair $(P_0,P_1)$ of the VIs \eqref{eq:vIs} such that $P_j(0,y)=J(j,y)$ for all $y\in(0,\infty)$.
Let us denote $\mathcal{Q}(t,y):=\big(P_0(t,y)-P_1(t,y)\big).$
Then, $\mathcal{Q}(t,y)$ satisfies the following parabolic double obstacle problem: on the domain $\Omega_T$
\begin{equation}\label{obs:prev}
    \begin{cases}
        \pat_t \mathcal{Q} +  \mathfrak{L}\mathcal{Q}+\mathcal{U}=0\qquad\text{for } -\phi_0(t)y < \mathcal{Q}(t,y) < \phi_1(t)y,\\
        \pat_t \mathcal{Q} +  \mathfrak{L}\mathcal{Q}+\mathcal{U}\le 0\qquad\text{for }\ \mathcal{Q}(t,y)=-\phi_0(t)y  ,\\
         \pat_t \mathcal{Q} +  \mathfrak{L}\mathcal{Q}+\mathcal{U}\ge 0\qquad\text{for }\ \mathcal{Q}(t,y)=\phi_1(t)y  ,\\
         -\phi_0 y\le \mathcal{Q} \le \phi_1y\ \ \text{in}\ \  \overline{\Omega_T},\quad \mathcal{Q}(T,y)=0 \ \ \text{for all}\ \ y\in(0,\infty),
    \end{cases}
\end{equation}
where $\mathcal{U}(y):=\big(\widetilde{u}_0(y)-\widetilde{u}_1(y)\big)+(\e_0-\e_1)y.$
Then, by the following theorem, we recover the solution $(P_0, P_1)$ to the system of VIs~\eqref{eq:vIs} from the obstacle problem~\eqref{obs:prev}:
\begin{thm}\label{thm:auxiliary_problem} (\cite[Theorem~3.4]{ZJ2025})
If $\mathcal{Q} \in W^{1,2}_{p,\mathrm{loc}}(\Omega_T) \cap C(\overline{\Omega_T})$ is the unique strong solution of the obstacle problem~\eqref{obs:prev} for any $p \ge 1$,  
then there exists a unique strong solution $(P_0, P_1)$ with 
$P_0, P_1 \in W^{1,2}_{p,\mathrm{loc}}(\Omega_T) \cap C(\overline{\Omega_T})$
of the following system:
\begin{align}\label{eq:auxiliary_problem}
\begin{cases}
 -\big(\partial_t P_0  + \mathfrak{L}P_0 + \widetilde{u}_0 + \e_0 y\big)
   = \big(-\partial_t \mathcal{Q} - \mathfrak{L}\mathcal{Q} -  \mathcal{U}\big)^+,\\
 -\big(\partial_t P_1 + \mathfrak{L}P_1 + \widetilde{u}_1 + \e_1 y\big)
   = \big(-\partial_t \mathcal{Q}  - \mathfrak{L}\mathcal{Q}  - \mathcal{U}\big)^-, \\
 P_0(T,y) = P_1(T,y) = J_R(y) \ \ \text{for all}\ \ 0<y<\infty.
\end{cases}
\end{align}\end{thm}
We can easily verify that the unique strong solution $(P_0, P_1)$ of~\eqref{eq:auxiliary_problem} satisfies the system of variational inequalities~\eqref{eq:vIs}. Thus, we now focus on analyzing the double obstacle problem~\eqref{obs:prev}.


\section{A Parabolic Double Obstacle Problem with Time-varying Obstacles}\label{sec:double}

In this section, we investigate the existence and uniqueness of the solution to the double obstacle problem~\eqref{obs:prev}, and derive estimates for the derivatives of the solution.

By applying the change of variables by setting\ \ $\tau=T-t$, \ $x=\ln y$\ \ and
\[
v(\tau,x):=e^{-x}\mathcal{Q}(T-\tau,e^x),
\]
\begin{equation}\label{U}
    U(x):=e^{-x}\mathcal{U}(e^x)=(\e_0-\e_1)-\frac{\g_1}{1-\g_1}(L_1^{\frac{\g_1-\g}{\g_1}}-L_0^{\tfrac{\g_1-\g}{\g_1}})e^{-\frac{1}{\g_1}x},
\end{equation}
\begin{equation}\label{psi}
\p_i(\tau):=\phi_i(T-\tau), \ \ i=0,1,\ \ \text{with}\ \ \|\p_0\|_{C^{1,1}([0,T])}+\|\p_1\|_{C^{1,1}([0,T])}=:q<\infty, 
\end{equation}
we can transform the above problem as, on $\mathcal{D}_T:=(0,T)\times \mathbb{R}$,
\begin{align}\label{obs}
    \begin{cases}
        \pat_\tau v - \mathcal{L}v = U 
        \quad \text{for }\ -\p_0(\tau)<v(\tau,x)<\p_1(\tau),
        \\
        \pat_\tau v - \mathcal{L}v \ge U 
        \quad \text{for } \ v(\tau,x)=-\p_0(\tau),
        \\
        \pat_\tau v - \mathcal{L}v \le U
        \quad \text{for }\  v(\tau,x)=\p_1(\tau),
        \\
        -\p_0 \le v \le \p_1 \ \ \text{in }\ \overline{\mathcal{D}_T},
        \qquad v(0,x)=0 \quad \text{for\ \ all }\ x\in\mathbb{R},
    \end{cases}
\end{align}
where 
\[
\mathcal{L}:=\frac{\t^2}{2}\pat_{xx}+(\b-r+\frac{\t^2}{2})\pat_x -r.
\]
We note that Assumption~\ref{as2} together with \eqref{eq:t1} implies the following properties for $\psi_0$ and $\psi_1$:
\begin{equation}\label{psi12}
-\partial_\tau \psi_0(\tau) - r\psi_0(\tau)  <  \partial_\tau \psi_1(\tau)+r\psi_1(\tau) < \e_0 - \e_1
   \quad \text{for all }\  \tau\in [0,T],
   \end{equation}
   and, letting $\tau_1:= T-t_1$,
\begin{equation}\label{eq:tau1}
\begin{cases}
\psi_1(\tau) > \mathcal{E}(\tau), & \tau\in [0, \tau_1), \\
\psi_1(\tau_1) = \mathcal{E}(\tau_1) ,&\\
\psi_1(\tau) < \mathcal{E}(\tau) & \tau\in (\tau_1,T], 
\end{cases}
\quad \text{where }\ \mathcal{E}(\tau):= \mathscr{E}(T-\tau)=(\e_0-\e_1)\frac{1-e^{-r\tau}}{r} .
\end{equation}
These will guarantee that the two free boundaries of the above problem exist, and remain separated.
\begin{rem}\label{E}
If $e^{-rt}(\mathscr{E}-\phi_1)$ were not decreasing,
there would exist a point where
$\pat_\tau \psi_1+r\psi_1\ge \e_0-\e_1$.
If a contact occurs at such a point, then \eqref{obs} yields
$\pat_\tau \psi_1+r\psi_1\le U(x)< \e_0-\e_1$,
a contradiction.
Hence Assumption~\ref{as2} allows the appearance of the region $\{v=\psi_1\}$.
\end{rem}

\subsection{Existence and Uniqueness}
To establish the existence and uniqueness of a solution to the double obstacle problem~\eqref{obs},
we approximate it by a sequence of double obstacle problems posed on bounded domains $\mathcal{D}_T^n:=\{(\tau,x):0<\tau<T,\ -n<x<n\}$ for each $n\in\mathbb{N}$.
For the lateral boundary conditions on $\partial\mathcal{D}_T^n\cap\{x=\pm n\}$,
we introduce the functions $\rho_0$ and $\rho_1$ defined by
    \begin{equation}\label{rho_01}
        \r_0(\tau)
        :=\max\{-\mathcal{G}(\tau),-\p_0(\tau)\}
\quad
\text{and}
\quad 
    \r_1(\tau)
    :=\min\{\mathcal{E}(\tau),\p_1(\tau)\}
    \end{equation}
for $\tau\in[0,T]$.
where  $\mathcal{E}(\tau)$ is given in \eqref{eq:tau1} and 
\[
\mathcal{G}(\tau):=\frac{\g_1}{1-\g_1}(L_1^\gam-L_0^\gam) e^{\frac{n}{\g_1}}\tau 
\]
The functions $\rho_0$ and $\rho_1$ are constructed to satisfy the compatibility condition, while $\mathcal{E}$ and $\mathcal{G}$ are designed so that $\rho_0 \le v_n \le \rho_1$ follows from the comparison principle.

With the above definitions, we formulate the following double obstacle problem: on $\mathcal{D}_T^n$,
\begin{equation}\label{obs:bdd}
    \begin{cases}
        \pat_\tau v_n  -\mathcal{L}v_n =U
        \quad \text{for }\ -\p_0(\tau)<v_n(\tau,x)<\p_1(\tau),
        \\
        \pat_\tau v_n -\mathcal{L}v_n\ge U
        \quad \text{for }\  v_n(\tau,x)=-\p_0(\tau),
        \\
        \pat_\tau v_n -\mathcal{L}v_n\le U
        \quad \text{for }\  v_n(\tau,x)=\p_1(\tau),
        \\
        -\p_0\le v_n\le \p_1 \ \ \text{in }\ \overline{\mathcal{D}_T^n}, \qquad v_n(0,x)=0 \ \ \text{for all }\ x\in[-n,n]
        ,
        \\
        v_n(\tau,-n)=\r_0(\tau) \quad \text{and} \quad
        v_n(\tau,n)=\r_1(\tau) \quad \text{for all }\ \tau\in[0,T].
    \end{cases}
\end{equation}
To obtain the existence of solution $v_n$ of \eqref{obs:bdd}, we approximate the above problem by the following penalized problem: for $\e >0$,
\begin{equation}\label{obs:pen}
    \begin{cases}
        \pat_\tau v_{n,\e}-\mathcal{L}v_{n,\e}=U+\b_{0,\e}(v_{n,\e}+\p_0)+\b_{1,\e}(v_{n,\e}-\p_1)
        \quad \text{in }\ \mathcal{D}_T^n,
        \\
        v_{n,\e}(0,x)=0 \quad \text{for all }\ x\in[-n,n],
        \\
        v_{n,\e}(\tau,-n)=\r_0(\tau)\quad \text{and} \quad
        v_{n,\e}(\tau,n)=\r_1(\tau) \quad \text{for all }\ \tau\in[0,T],
    \end{cases}
\end{equation}
where $\b_{0,\e},\b_{1,\e}\in C^\infty(\mathbb{R})$ are called the penalty functions, which satisfy
\begin{equation*}
    \begin{cases}
        \b_{0,\e}(0)=\frac{\g_1}{1-\g_1}
        (L_1^{\gam}-L_0^{\gam})e^{\frac{n}{\g_1}}+q,
        \\
        \b_{0,\e}(\l)\ge 0 \quad \text{for all }\ \l\in\mathbb{R},
        \\
        \b_{0,\e}(\l)=0 \quad \text{if }\ \l\ge \e,
        \\
        \b_{0,\e}'(\l)\le 0 \quad \text{for all }\ \l\in\mathbb{R},
    \end{cases}
    \qquad
    \begin{cases}
        \b_{1,\e}(0)=-(\e_0-\e_1)-q,
        \\
        \b_{1,\e}(\l)\le 0\quad \text{for all }\ \l\in\mathbb{R},
        \\
        \b_{1,\e}(\l)=0 \quad \text{if }\ \l\le -\e,
        \\
        \b_{1,\e}'(\l)\le 0 \quad \text{for all }\ \l\in\mathbb{R},
    \end{cases}
\end{equation*}
where the constant $q$ is given in \eqref{psi}.

Under this construction, we begin by investigating the existence, uniqueness and the estimate for the solution $v_{n,\e}$ of \eqref{obs:pen} and consequently obtain the existence and uniqueness of the solution $v_n$ of \eqref{obs:bdd}.
\begin{lem}\label{ineq:pen}
    For each small $\e>0$ and $n\in\mathbb{N}$, there exists a unique solution $v_{n,\e}\in W^{1,2}_p(\mathcal{D}_T^n)\cap C(\overline{\mathcal{D}_T^n})$, $1<p<\infty$ to the problem \eqref{obs:pen}.
In particular, $v_{n,\e}$ satisfies the following estimate:
    \begin{equation}\label{ineq:v,rho}
        -\p_0\le \r_0 \le v_{n,\e}\le \r_1 \le \p_1 \quad 
        \text{in }\ \mathcal{D}_T^n.
    \end{equation}
\end{lem}
\begin{proof}    
    The existence of the solution $v_{n,\e}$ follows from the Schauder fixed point theorem \cite[Theorem 8.3]{Lieberman}, while the uniqueness of $v_{n,\e}$ is guaranteed by the comparison principle \cite[Theorem 1.1]{YangYan2008}.
    
    We prove the estimate $\eqref{ineq:v,rho}$ using the comparison principle \cite[Theorem 1.1]{YangYan2008}.
Note that the functions $\r_0$ and $\r_1$ in \eqref{rho_01} satisfy
    \begin{align*}
        \pat_\tau \r_0 -\mathcal{L}\r_0 
            &-\b_{0,\e}(\r_0+\p_0)-\b_{1,\e}(\r_0-\p_1)
            \\
            & = \bigg[-\tfrac{\g_1}{1-\g_1}
            (L_1^{\gam}-L_0^{\gam}) e^{\frac{n}{\g_1}}(1+r\tau)-\b_{0,\e}(-\mathcal{G}+\p_0)\bigg]
            \mathbf{1}_{\{-\mathcal{G}> -\p_0\}}
            \\
            &\quad -(\p'_0+r\p_0+\b_{0,\e}(0))
            \mathbf{1}_{\{-\mathcal{G} \le -\p_0\}}
            \\
            & \le -\tfrac{\g_1}{1-\g_1}(L_1^{\gam}-L_0^{\gam}) e^{\frac{n}{\g_1}}(1+r\tau)\mathbf{1}_{\{-\mathcal{G}> -\p_0\}}
            \\
            &\quad \left[-\tfrac{\g_1}{1-\g_1}(L_1^{\gam}-L_0^{\gam}) e^{\frac{n}{\g_1}}-\p'_0-q\right]
            \mathbf{1}_{\{-\mathcal{G}> -\p_0\}}
            \\
            & \le -\tfrac{\g_1}{1-\g_1}(L_1^{\gam}-L_0^{\gam}) e^{\frac{n}{\g_1}} \le U
    \end{align*}
    and
        \begin{align*}
        \pat_\tau \r_1 -\mathcal{L}\r_1 
            &-\b_{0,\e}(\r_1+\p_0)-\b_{1,\e}(\r_1-\p_1)
            \\
            & \ge \Big[(\e_0-\e_1)-\b_{1,\e}(\mathcal{E}-\p_1)\Big]\mathbf{1}_{\{\mathcal{E}\le\p_1\}}
            +\{\p'_1+r\p_1-\b_{1,\e}(0)\}
            \mathbf{1}_{\{\mathcal{E}>\p_1\}}
            \\
            &=\Big[(\e_0-\e_1)-\b_{1,\e}(\mathcal{E}-\p_1)\Big]\mathbf{1}_{\{\mathcal{E}\le\p_1\}}
            +\{\p'_1+r\p_1+\e_0-\e_1+q\}
            \mathbf{1}_{\{\mathcal{E}>\p_1\}}
            \\
            & \ge \e_0-\e_1 \ge U
    \end{align*}
    in $\mathcal{D}_T^n$, provided that $\e < \min_{\tau\in[0,T]}\{\psi_0(\tau)+\psi_1(\tau)\}$.
    Thus, $\r_0$ satisfies
    \begin{align*}
        \begin{cases}
            \pat_\tau \r_0 -\mathcal{L}\r_0 
            -\b_{0,\e}(\r_0+\p_0)-\b_{1,\e}(\r_0-\p_1)
            \le U \quad &\text{in }\ \mathcal{D}_T^n,
            \\
            \r_0(0)\le 0 = v_{n,\e}(0,x)
            \quad &\text{for all }\ x\in[-n,n],
            \\
            \r_0(\tau)\le v_{n,\e}(\tau,n)
            \quad \text{and } \quad
            \r_0(\tau) = v_{n,\e}(\tau,-n)
            \quad &\text{for all }\ \tau\in[0,T].
        \end{cases}
    \end{align*}
    Similarly, $\r_1$ satisfies
    \begin{align*}
        \begin{cases}
            \pat_\tau \r_1 -\mathcal{L}\r_1 
            -\b_{0,\e}(\r_1+\p_0)-\b_{1,\e}(\r_1-\p_1)
            \ge U \quad &\text{in }\ \mathcal{D}_T^n,
            \\
            \r_1(0)\ge 0 = v_{n,\e}(0,x)
            \quad &\text{for all }\ x\in[-n,n],
            \\
            \r_1(\tau)= v_{n,\e}(\tau,n)
            \quad \text{and} \quad
            \r_1(\tau) \ge v_{n,\e}(\tau,-n)
            \quad &\text{for all }\ \tau\in[0,T].
        \end{cases}
    \end{align*}
    Therefore, by the comparison principle for the non-linear equation \cite[Theorem 1.1]{YangYan2008}, we obtain the inequality \eqref{ineq:v,rho}.
\end{proof}

 Now, we prove the existence and the uniqueness of the solution to the variational inequality \eqref{obs}.
\begin{lem}\label{lem:v_existence}
    For each $n\in\mathbb{N}$, there exists a unique solution $v_n\in W^{1,2}_p(\mathcal{D}_T^n)\cap C(\overline{\mathcal{D}_T^n})$ to \eqref{obs:bdd}. Furthermore, there exists a unique solution $v\in W^{1,2}_p(\mathcal{D}_T)\cap C(\overline{\mathcal{D}_T})$ to \eqref{obs} for $1 < p < \infty$.
\end{lem}

\begin{proof}
Since the existence of such solutions follows from standard arguments in the literature (see, e.g., \cite{HaJeonOk2025Chooser} for further details), we only provide a brief sketch of the proof. 
    
    The existence of $v_n$ is established by a penalization method. Let $v_{n,\varepsilon}$ be the solution to the regularized problem \eqref{obs:pen} with penalty terms $\beta_{0,\varepsilon}$ and $\beta_{1,\varepsilon}$. Using the $W^{1,2}_p$-estimates \cite[Theorem 9.1, p.341]{ladyzhenskaya}, $\Vert v_{n,\varepsilon} \Vert_{W^{1,2}_p(\mathcal{D}_T^n)}$ is uniformly bounded with respect to $\varepsilon$ due to the monotonicity of the penalty terms. Passing to the limit $\varepsilon \to 0$, we obtain $v_n$ satisfying $-\psi_0 \le v_n \le \psi_1$ and the corresponding variational inequalities.

    For the solution $v$ on the unbounded domain $\mathcal{D}_T$, we apply interior $W^{1,2}_p$-estimates \cite[p.355]{ladyzhenskaya}. For any fixed $N \in \mathbb{N}$ and $n \ge 2N$, $\Vert v_n \Vert_{W^{1,2}_p(\mathcal{D}_T^N)}$ is uniformly bounded since $\|v_n\|_{L^p(\mathcal{D}^{2N}_T)}$ and $\|\partial_\tau v_n - \mathcal{L}v_n\|_{L^p(\mathcal{D}^{2N}_T)}$ are bounded independently of $n$. A diagonal argument then yields a subsequence $\{v_n\}$ converging to $v \in W^{1,2}_{p,\text{loc}}(\mathcal{D}_T)$ weakly and uniformly on compact sets. Taking $n \to \infty$ ensures that $v$ satisfies \eqref{obs}. Uniqueness follows from the comparison principle \cite[Theorem 1.2]{YangYan2008}.
\end{proof}

In the following lemmas, we establish several estimates for the derivatives of the solution $v$.
We first show that $v(\tau,\cdot)$ is monotone increasing for every $\tau\in(0,T)$.
This property allows us to parametrize each free boundary, $\partial\{v<\p_1\}$ and $\partial\{v>-\p_0\}$, with respect to~$\tau$.
We then derive an estimate for $\partial_\tau v + (-1)^i \p_i'$, $i=0,1$, which will play an essential role in obtaining a key estimate later on.

\begin{lem}\label{estimate(1)}
    Let $v$ be the solution to \eqref{obs}. Then,  $\pat_x v \ge 0$ in $\mathcal{D}_T$. In particular, 
\[\pat_x v >0 \quad \text{in }\ \{-\p_0<v<\p_1\}.\]
\end{lem}
\begin{proof}
    Let us consider the solution $v_{n,\e}$ of the problem \eqref{obs:pen}.
    Recall that $v_{n,\e}\in W^{1,2}_p(\mathcal{D}_T^n)$, $1<p<\infty$. By the Sobolev embedding theorem \cite[80p, Lemma 3.3]{ladyzhenskaya}, $v_{n,\e} \in C^{\frac{\alpha}{2}, \alpha}(\overline{\mathcal{D}_T^n})$, $\alpha \in (0,1)$.
    Thus, the right hand side of the equation in \eqref{obs:pen} is of class $C^{\frac{\a}{2},\a}$. By the Schauder theory, $v_{n,\e}\in C^{1+\frac{\a}{2},2+\a}_{\rm loc}(\mathcal{D}_T^n)$. Again, the right hand side becomes $C^{1+\frac{\a}{2},2+\a}$ and by the Schauder theory, we obtain that $v_{n,\e}\in C^{2+\frac{\a}{2},4+\a}_{\mathrm{loc}}(\mathcal{D}_T^n)$. This allows us to differentiate the equation in \eqref{obs:pen} with respect to $x$. Thus, differentiating the equation together with the inequality \eqref{ineq:v,rho} yields that $\partial_x v_{n,\e}$ satisfies
    \begin{equation}\label{partial x}
        \begin{cases}
        \pat_\tau (\pat_x v_{n,\e})
        -\mathcal{L}(\pat_x v_{n,\e})
        -\{\b'_{0,\e}(\cdots)+\b'_{1,\e}(\cdots)\}\pat_x v_{n,\e}
        =U'
        \quad \text{in }\ \mathcal{D}_T^n,
        \\
        \pat_x v_{n,\e}(\tau,-n)\ge 0
        \quad \text{and} \quad
        \pat_x v_{n,\e}(\tau,n)\ge 0 \quad \text{for }\ \text{all }\ \tau\in[0,T],
        \\
        \pat_x v_{n,\e}(0,x)=0
        \quad \text{for }\ \text{all}\ \ x\in[-n,n].
    \end{cases}
    \end{equation}
    Since $U'(x)>0$ for all $x\in \mathbb{R}$, we deduce by the maximum principle that $\pat_x v_{n,\e}\ge 0$ in $\mathcal{D}_T^n$.
    For all small $\e>0$ and large $n\in \mathbb{N}$,
    since $\partial_x v$ is a weak limit of a subsequence of $\{\partial_x v_{n,\e}\}$,
    we obtain that $\partial_x v\ge 0$.

    To prove the strict positivity of $\pat_x v$ in the region $\{-\p_0<v<\p_1\}$ we fix an arbitrary point
    $(\tau_0,x_0)\in\{-\p_0<v<\p_1\}$.
    Since $U$ is a smooth function, by the Schauder theory, there exists a neighborhood $N:=B_\d(\tau_0,x_0)$ contained in the set $\{-\p_0<v<\p_1\}$ for some $\d>0$ such that $v$ is smooth in $N$. 
    Then, $v$ satisfies $\pat_\tau v -\mathcal{L}v=U$ in $N$ and
    by differentiating the equation with respect to $x$, we have 
    \[\pat_\tau (\pat_x v) -\mathcal{L}(\pat_x v) = U' >0
    \quad \text{in }\ N'=B_{\d/2}(\tau_0,x_0).\]
    If $\pat_x v(\tau_0,x_0)=0$, we deduce by the strong minimum principle that $\pat_x v\equiv 0$ in $N'$, which is a contradiction. Hence $\pat_x v(\tau_0,x_0)>0$. Since the point $(\tau_0,x_0)$ is arbitrary, we conclude $\pat_x v>0$ in $\{-\p_0<v<\p_1\}$.
    \end{proof}

\begin{lem}\label{estimate(2)}
Let $v$ be the solution to \eqref{obs}. Then, for each $i=0,1$, the following estimate holds for a.e. $(\tau,x)\in\mathcal{D}_T$:
\begin{equation*}
 -\tfrac{\gamma_1}{1-\gamma_1}
   (L_1^{\gam}-L_0^\gam)e^{K_1(T-\tau)-\frac{x}{\gamma_1}}
 - 2q
 \le
 \partial_\tau v(\tau,x) + (-1)^i \psi_i'(\tau)
 \le
 (\e_0 - \e_1)e^{-r\tau} + 2q,
\end{equation*}
where the constants $q$ and $K_1$ are defined in 
Assumption~\ref{as4} and \eqref{psi}, respectively.
\end{lem}
    \begin{proof}
    As described in the proof of the previous lemma, we can eventually see that $v_{n,\e}$ is infinitely differentiable in $\mathcal{D}_T^n$. Therefore, differentiating the equation in \eqref{obs:pen} with respect to $\tau$ yields the following.
    \[\pat_{\tau\tau} v_{n,\e}
        -\mathcal{L}\pat_\tau v_{n,\e}
        -\b'_{0,\e}(\cdots)
        (\pat_\tau v_{n,\e}+\p'_0)
        -\b'_{1,\e}(\cdots)
        (\pat_\tau v_{n,\e}-\p'_1)=0
        \quad \text{in }\ \mathcal{D}_T^n.\]
        Since $v_{n,\e}(0,\cdot) = 0$, it follows by the result in \cite[Theorem 19, p.321]{Friedman} that $v_{n,\e}$ is smooth up to the initial boundary. 
        This implies that for each $x\in[-n,n]$,
    \begin{align*}
        \pat_\tau v_{n,\e}(0,x)
        &=\mathcal{L}v_{n,\e}(0,x)
        +\b_{0,\e}(v_{n,\e}(0,x)+\p_0(0))
        +\b_{1,\e}(v_{n,\e}(0,x)-\p_1(0))
        +U(x)
        \\
        &=\b_{0,\e}(\p_0(0))
        +\b_{1,\e}(-\p_1(0))
        +U(x) =U(x)
    \end{align*}
    for  sufficiently small  $\e<\min\{\p_0(0),\p_1(0)\}$.
    
    Combining all, we find that 
    $\pat_\tau (v_{n,\e}+\p_0)$ satisfies
    \begin{equation}\label{partial tau}
        \begin{cases}
        \pat_{\tau\tau} (v_{n,\e}+\p_0)
        -\mathcal{L}
        \pat_\tau (v_{n,\e}+\p_0)
        -\{\b'_{0,\e}(\cdots)+\b'_{1,\e}(\cdots)\}
        \pat_\tau (v_{n,\e}+\p_0)
        \\
        =\p''_0+r\p'_0 
        -\b'_{1,\e}(\cdots)
        \big( \p'_0+\p'_1 \big)
        \quad \text{in }\ \mathcal{D}_T^n,
        \\
        \pat_\tau (v_{n,\e}+\p_0)(0,x)=U(x)+\p'_0(0) \quad \text{for }\ \text{all}\ \ x\in[-n,n],
        \\
        \pat_\tau(v_{n,\e}+\p_0)(\tau,-n)
        =\pat_\tau \r_0(\tau)+\p'_0(\tau) \quad \text{and} 
        \\
        \pat_\tau (v_{n,\e}+\p_0)(\tau,n)
        =\pat_\tau \r_1(\tau)+\p'_0(\tau)\quad 
        \text{for }\ \text{all }\ \tau\in[0,T].

        \end{cases}
    \end{equation}
       Define $\overline{v}(\tau):=
    (\e_0-\e_1)e^{-r\tau}+\p'_0(\tau)+q>0$. Then $\overline{v}$ satisfies that, using $\b_{0,\e}',\b_{1,\e}'\le 0$, 
    \begin{align*}
        &\pat_\tau \overline{v} -\mathcal{L}\overline{v}
        -\b'_{0,\e}(\cdots)\overline{v}-\b'_{1,\e}(\cdots)\overline{v}
        \\
        &= \p''_0+r(\p'_0+q)-\{\b'_{0,\e}(\cdots)+\b'_{1,\e}(\cdots)\}(\p'_0+q)
        \\
        & \ge\p''_0+r\p'_0 
        -\b'_{1,\e}(\cdots)(\p'_0+\p'_1)
    \end{align*}
    in $\mathcal{D}_T^n$.    
    Moreover, on the parabolic boundary of $\mathcal{D}^n_T$, direct computations yield that
     \[
        \overline{v}(0,x) \ge U(x)+\p'_0(0) 
        \quad \text{for all }\ x\in[-n,n], 
\]
     \[
             \overline{v}(\tau, -n) \ge \partial_\tau \rho_0(\tau) + \p'_0(\tau)
   \ \ \text{and}\ \ 
        \overline{v}(\tau, n)\ge \partial_\tau \r_1(\tau) + \p'_0(\tau)
        \quad \text{for all } \ \tau\in[0,T].
 \]
    Thus, by the comparison principle \cite[Theorem 1.1]{YangYan2008}, we deduce that
    \[\pat_\tau(v_{n,\e}+\p_0)\le \overline{v}
    =(\e_0-\e_1)e^{-r\tau}+\p'_0+q\quad \text{in }\  \mathcal{D}_T^n.\]
In a similar way, we also find that $\pat_\tau(v_{n,\e}-\p_1)$ satisfies 
    \begin{equation}\label{tau,psi,0}
        \begin{cases}
        \pat_{\tau\tau} (v_{n,\e}-\p_1)
        -\mathcal{L}
        \pat_\tau (v_{n,\e}-\p_1)
        -\{\b'_{0,\e}(\cdots)+\b'_{1,\e}(\cdots)\}
        \pat_\tau (v_{n,\e}-\p_1)
        \\
        =-\p''_1-r\p'_1 
        -\b'_{0,\e}(\cdots)
        \big( -\p'_0-\p'_1 \big)
        \quad \text{in }\ \mathcal{D}_T^n,
        \\
        \pat_\tau(v_{n,\e}-\p_1)(\tau,-n)
        =\pat_\tau \r_0(\tau)-\p'_1(\tau) \quad \text{and} 
        \\
        \pat_\tau (v_{n,\e}-\p_1)(\tau,n)
        =\pat_\tau \r_1(\tau)-\p'_1(\tau)
        \quad \text{for }\ \text{all }\ \tau\in[0,T],
        \\
        \pat_\tau (v_{n,\e}-\p_1)(0,x)=U(x)-\p'_1(0) \quad \text{for }\ \text{all}\ \ x\in[-n,n].
        \end{cases}
    \end{equation}
Hence, by applying the comparison principle again,   we obtain $\pat_\tau (v_{n,\e}-\p_1)\le (\e_0-\e_1)e^{-r\tau}-\p'_1+q$ .
    Therefore, recalling the definition of $q$ in \eqref{psi}, we conclude that
    \[\pat_\tau v_{n,\e}+(-1)^i\p'_i
    \le (\e_0-\e_1)e^{-r\tau}+2q, \quad i=0,1
    \quad \text{in }\ \mathcal{D}_T^n.\]
    
   We next prove the first inequality in  the lemma. Define \[\underline{v}(\tau,x):=-\tfrac{\g_1}{1-\g_1}(L_1^{\gam}-L_0^{\gam})e^{K_1(T-\tau)-\frac{1}{\g_1}x}
    +\p'_0(\tau)-q.\]
    Then, since $-\mathcal{L}e^{-\frac{1}{\g_1}x}=K_1e^{-\frac{1}{\g_1}x}$, $\underline{v}$ satisfies that in $\mathcal{D}_T$,
    \begin{align*}
        &\pat_\tau\underline{v} -\mathcal{L}\underline{v}
        -\b'_{0,\e}(\cdots)\underline{v}-\b'_{1,\e}(\cdots)\underline{v}
        \\
        &\le  \p''_0+r(-q+\p'_0)
        -\{\b'_{0,\e}(\cdots)+\b'_{1,\e}(\cdots)\}(-q+\p'_0)
        \\
        & \le \p''_0+r\p'_0
        -\b'_{1,\e}(\cdots)(\p'_0+\p'_1)
    \end{align*}
    since  $\b_{0,\e}'\le 0$ and $\b_{1,\e}'\le 0$.
    Moreover, on the parabolic boundary of $\mathcal{D}^n_T$, direct computations yield that
     \[
        \underline{v}(0,x) \le U(x)+\p'_0(0) 
        \quad \text{for all }\ x\in[-n,n], 
\]
     \[
             \underline{v}(\tau, -n) \le \partial_\tau \rho_0(\tau) + \p'_0(\tau)
   \ \ \text{and}\ \ 
        \underline{v}(\tau, n)\le \partial_\tau \r_1(\tau) + \p'_0(\tau)
        \quad \text{for all } \ \tau\in[0,T].
 \]
    Thus, by the comparison principle \cite[Theorem 1.1]{YangYan2008}, we deduce that
    \[-\tfrac{\g_1}{1-\g_1}(L_1^{\gam}-L_0^{\gam})e^{K_1(T-\tau)-\frac{1}{\g_1}x} +\p'_0-q
    \le \pat_\tau(v_{n,\e}+\p_0)\quad \text{in }\  \mathcal{D}_T^n.\] 
    Furthermore, by following the same argument, we can see that $\pat_\tau(v_{n,\e}-\p_1)$ satisfies
    \[-\tfrac{\g_1}{1-\g_1}(L_1^{\gam}-L_0^{\gam})e^{K_1(T-\tau)-\frac{1}{\g_1}x} -\p'_1-q
    \le \pat_\tau(v_{n,\e}-\p_1).\] 
    Therefore, we conclude
    \[ -\tfrac{\g_1}{1-\g_1}(L_1^{\gam}-L_0^{\gam})e^{K_1(T-\tau)-\frac{1}{\g_1}x}-2q\le \pat_\tau v_{n,\e}+(-1)^i\p'_i
    \quad i=0,1\ \text{ in }\ \mathcal{D}_T^n.\]
    
    Finally, from the conclusion that for $i=0,1$,
    \[-\tfrac{\g_1}{1-\g_1}(L_1^{\gam}-L_0^{\gam})e^{K_1(T-\tau)-\frac{1}{\g_1}x} -2q
    \le \pat_\tau v_{n,\e}+(-1)^i\p'_i
    \le (\e_0 - \e_1)e^{-r\tau}+2q,\]
    weak convergence of $\pat_\tau v_{n,\e}$ to $\pat_\tau v$ as $\e\to 0^+$ and $n\to \infty$ yields the result.
\end{proof}

\section{Free boundary analysis}
\subsection{Existence and properties of the free boundaries}
In this subsection, we verify that both switching regions
$\{v=-\psi_0\}$ and $\{v=\psi_1\}$ are nonempty, ensuring the
existence of the corresponding free boundaries
$\partial\{v=-\psi_0\}$ and $\partial\{v=\psi_1\}$.
We also identify the regions in the domain $\mathcal{D}_T$, where the
switching is expected to occur and analyze the geometric structure of
the contact sets.

To this end, we construct suitable auxiliary functions that solve
appropriate variational inequalities respectively, and apply the comparison principle
\cite[Theorem 1.2]{YangYan2008} to deduce that the switching regions are nonempty.
This analysis allows us to localize the contact sets and to describe
their boundaries through suitable max–min characterizations of the solution.

We begin by defining the switching regions $\mathcal{SR}_0$ and
$\mathcal{SR}_1$ as follows:
\begin{equation}
        \mathcal{SR}_0:=\{(\tau,x)\in\mathcal{D}_T 
        : v(\tau,x) = -\p_0(\tau) \},
\end{equation}
\begin{equation}
        \mathcal{SR}_1:=\{(\tau,x)\in\mathcal{D}_T 
        : v(\tau,x) = \p_1(\tau) \}.
\end{equation}
Furthermore, define the curves $\Gamma_0(\tau)$ and $\Gamma_1(\tau)$ by
    \[\Gamma_0(\tau):=\g_1\ln\bigg( \frac{\g_1}{1-\g_1} \frac{L_1^{\gam}-L_0^{\gam}}{\e_0-\e_1+\pat_\tau\p_0(\tau) +r\p_0(\tau)} \bigg), \quad \tau\in [0,T],\]
    \[\Gamma_1(\tau):=\g_1\ln\bigg( \frac{\g_1}{1-\g_1} \frac{L_1^{\gam}-L_0^{\gam}}{\e_0-\e_1-\pat_\tau \p_1(\tau)-r\p_1(\tau)} \bigg),
    \quad \tau\in[\tau_1,T].\]
    Note that by \eqref{psi12}, $\Gamma_0$ and $\Gamma_1$ are well-defined, and $\Gamma_1>\Gamma_0$.
\begin{lem}\label{reg:s_1}   
Define the regions $\mathcal{D}_0$ and $\mathcal{D}_1$ by
\[
\mathcal{D}_0 := \{(\tau,x)\in(0,T)\times\mathbb{R} : x \le \Gamma_0(\tau)\},
\qquad
\mathcal{D}_1 := \{(\tau,x)\in(\tau_1,T)\times\mathbb{R} : x \ge \Gamma_1(\tau)\},
\]
where the constant $\tau_1$ is given in \eqref{eq:tau1}.
Then, $\mathcal{SR}_0\subset \mathcal{D}_0$ and $\mathcal{SR}_1\subset\mathcal{D}_1$.
\end{lem}
\begin{proof}
    Since $v\in W^{1,2}_{p,{\rm loc}}(\mathcal{D}_T)$ for $1<p<\infty$, it follows that 
    \begin{equation*}
        -\pat_\tau \p_0 -r\p_0
        =\pat_\tau v -\mathcal{L}v
        \ge U,
    \end{equation*}
    almost everywhere in $\{v=-\p_0\}$.
    Therefore, we see that for almost every $(\tau,x)\in\{v=-\p_0\}$, $x$ satisfies the following inequality:
    \[x\le \g_1\ln\left( \frac{\g_1}{1-\g_1} \frac{L_1^{\gam}-L_0^{\gam}}{\e_0-\e_1+\pat_\tau \p_0(\tau)+r\p_0(\tau)} \right)=\Gamma_0(\tau), \quad \]
    This implies $\mathcal{SR}_0\subset \mathcal{D}_0$. 

    In order to prove $\mathcal{SR}_1\subset \mathcal{D}_1$,
    we first claim that 
    \begin{equation}\label{claim 0}
        v<\p_1 \quad \text{in }\ [0,\tau_1]\times \mathbb{R}.
    \end{equation}
    which implies $\mathcal{SR}_1\subset (\tau_1,T)\times\mathbb{R}$.
    Recall the inequality \eqref{ineq:v,rho} and take $\e\to 0^+$ and $n\to \infty$ to get $v\le \r_1$ in $\mathcal{D}_T$. By \eqref{eq:tau1} and the definition of $\rho_1$ in \eqref{rho_01}, we have 
    $v<\p_1 \quad\text{in }\ [0,\tau_1)\times \mathbb{R}$.
    Suppose that there exists $b_1\in \mathbb{R}$ such that
    $v(\tau_1,b_1)=\p_1(\tau_1)$.
    We first observe that
    $\pat_\tau v - \mathcal{L} v =U$ in $(0,\tau_1)\times(\max\limits_{\tau\in[0,T]}\Gamma_0(\tau),\infty)$ since $v<\p_1$ in $(0,\tau_1)\times \mathbb{R}$ and  $v>-\p_0$ in $(0,T)\times (\max\limits_{\tau\in[0,T]}\Gamma_0(\tau),\infty)$.
    Moreover, since $\r_1=\mathcal{E}$ in $(0,\tau_1)\times\mathbb{R}$,
    $\pat_\tau \r_1-\mathcal{L}\r_1 =\e_0-\e_1$
    in $(0,\tau_1)\times\mathbb{R}$.
    Combining the two previous equations implies that
    \begin{equation}\label{rho,v}
        \pat_\tau (\r_1-v)-\mathcal{L}(\r_1-v) = \e_0-\e_1-U>0 
        \quad \text{in }\ (0,\tau_1)\times(\max\limits_{\tau\in[0,T]}\Gamma_0(\tau),\infty).
    \end{equation}
    By the monotonicity of $v(\tau_1,\cdot)$, $\r_1(\tau_1)=\p_1(\tau_1)=v(\tau_1,x)$ for all $x\ge b_1$.
    This together with
    \eqref{ineq:v,rho} implies that for $M:=\max\{b_1,\max\limits_{\tau\in[0,T]}\Gamma_0(\tau)\}$, $\r_1-v$ attains its minimum value zero on the line $\{\tau_1\}\times(M,\infty)$.
    By applying the strong maximum principle \cite[Theorem 2.7]{Lieberman} to $\r_1- v$ on $(0,\tau_1)\times(M,\infty)$,
    we deduce that $\r_1\equiv v$ in the same region.
    However, this contradicts \eqref{rho,v}.
    Therefore, we prove the claim \eqref{claim 0}.

    It remains to show that $\mathcal{SR}_1\subset  \{(\tau,x)\in(0,T)\times\mathbb{R} : x \ge \Gamma_1(\tau)\}$.
    Since $v\in W^{1,2}_{p,{\rm loc}}(\mathcal{D}_T)$ for $1<p<\infty$, it follows that 
        \[\pat_\tau \p_1 +r\p_1
        =\pat_\tau v -\mathcal{L}v
        \le U,\]
    almost everywhere in $\{v=\p_1\}$.
    Therefore, we see that for almost every $(\tau,x)\in\{(\tau,x)\in\mathcal{D}_T : v(\tau,x)=\p_1(\tau)\}$, $x$ satisfies the following inequality:
    \[x\ge \g_1\ln\bigg( \frac{\g_1}{1-\g_1} \frac{L_1^{\gam}-L_0^{\gam}}{\e_0-\e_1-\pat_\tau \p_1(\tau) -r\p_1(\tau)} \bigg)=\Gamma_1(\tau).\]
    This implies $\mathcal{SR}_1\subset \{(\tau,x): x\ge \Gamma_1(\tau),\ \tau\in(0,T)\}$. 
    Combining the above results, we conclude that $\mathcal{SR}_1\subset \mathcal{D}_1$.
\end{proof}

Note that the previous result was obtained under the assumption that the switching regions $\mathcal{SR}_0$ and $\mathcal{SR}_1$ are nonempty.
We now show that both switching regions are indeed nonempty.
Moreover, we verify that the solution $v$ comes into contact with each obstacle, $-\p_0$ and $\p_1$, at every time $\tau \in (0,T)$ and $\tau \in (\tau_1,T)$, respectively.
\begin{lem}\label{reg:s_2}
    Let $v$ be the solution to the problem \eqref{obs}. Then, for each small $c>0$, there exists $K_c>0$ and  $X_c>0$ satisfying
     \[v=-\p_0 \quad \text{in }\ [c,T]\times(-\infty,-X_c],\]
    and \[v=\p_1 \quad \text{in }\ 
    [\tau_1+c,T]\times[K_c,\infty).\]
\end{lem}
\begin{proof}
    For each small $\e>0$ define function $\overline{z}$ in $[0,T]\times(-\infty,-1/\e^2]$ by
    \[\overline{z}(\tau,x):=
    -\p_0(\tau)\Big( \frac{\tau}{\e}\mathbf{1}_{\{\tau\le \e\}}+\mathbf{1}_{\{\tau> \e\}} \Big)+\e^3\Big( x+\frac{2}{\e^2} \Big)^2\mathbf{1}_{\{x \ge -2/{\e^2}\}}\ge -\p_0(\tau).\]
    Then, $\overline{z}\in W^{1,2}_{p,{\rm loc}}((0,T)\times(-\infty,-1/\e^2))$, $1<p<\infty$, and $\overline{z}$ satisfies that for all $(\tau,x)\in (0,T)\times(-\infty,-1/\e^2)$
    \begin{align*}
        \pat_\tau \overline{z}(\tau,x) - \mathcal{L} \overline{z}(\tau,x)
    &=-\left[\frac{\tau}{\e}\left\{\p'_0(\tau)+r\p_0(\tau)\right\} +\frac{1}{\e}\p_0(\tau)\right]\mathbf{1}_{\{\tau\le \e\}}
    \\
    &\quad-\big\{\p'_0(\tau)+r\p_0(\tau)\big\}\mathbf{1}_{\{\tau>\e\}}
    \\
    &\quad -\e^3\Big[ \t^2 + (2\b - 2r + \t^2)\Big(x + \frac{2}{\e^2}\Big) - r\Big(x + \frac{2}{\e^2} \Big)^2 \Big] \mathbf{1}_{\{x \ge -2/{\e^2}\}}\\
    &\ge  
    -q(r+1)-\frac{q}{\e}
    -\e^3\Big( \t^2 + \frac{2\b+\t^2}{\e^2} \Big) \mathbf{1}_{\{x \ge -2/{\e^2}\}}\\
    &\ge (\e_0-\e_1)-(L_1^\gam-L_0^\gam)\frac{\g_1}{1-\g_1}e^{\frac{1}{\g_1\e^2}} \ge U(x),
    \end{align*}
    where we used the fact that $\Vert \psi_0\Vert_{C^{1,1}([0,T])}\le q$ in \eqref{psi} in the first inequality and chose $\e>0$ small to hold the second inequality.

    Therefore,
    for each $(\tau,x)\in (0,T)\times(-\infty,-1/\e^2)$, 
    with $\e>0$ small enough to satisfy $\min\limits_{\tau\in[0,T]}\Gamma_1(\tau)>-1/\e^2$,
    we have $v<\p_1$ and
    $\overline{z}$ satisfies
\begin{equation*}
    \begin{cases}
        \pat_\tau \overline{z} (\tau,x)
        - \mathcal{L} \overline{z}(\tau,x)\ge U(x)
        \quad \text{for }\ (\tau,x)\in(0,T)\times(-\infty,-1/{\e^2}),
        \\
        \overline{z}(\tau, -1/{\e^2}) \ge -\p_0(\tau) + 1/\e \ge \p_1(\tau) \ge v(\tau, -1/{\e^2}) \quad \text{for all }\ \tau\in[0,T],
        \\
        \overline{z}(0, x) \ge 0 = v(0,x)
        \quad \text{for all }\ x\in (-\infty,-1/{\e^2}].
    \end{cases}
\end{equation*}
Applying the comparison principle for the variational inequality \cite[Theorem 1.2]{YangYan2008} to $\overline{z}$ and $v$ in the domain $(0, T) \times (-\infty,-1/\e^2)$, we deduce
\[\overline{z} \ge v \quad \text{in } [0, T] \times (-\infty,-1/\e^2].
\]
Since $v\ge -\p_0$ and $\overline{z}=-\p_0$ in $[\e, T) \times (-\infty,-2/{\e^2}]$, we conclude that
\[v(\tau, x) = -\p_0(\tau) \quad \text{for all }\ (\tau,x) \in [\e, T] \times (-\infty,-2/{\e^2}].
\]
This proves the first assertion with $c=\e$ and $X_c=2/{\e^2}$.

On the other hand, for each $\e>0$ and $(\tau,x)\in[0,T]\times[1/{\e^3},\infty)$, define $\underline{z}$ by
\begin{equation*}
    \underline{z}(\tau,x):=\min\{(1-\e)\mathcal{E}(\tau),\p_1(\tau)\}
    -\e^5\left( x-\frac{2}{\e^3} \right)^2
    \mathbf{1}_{\{x\le \frac{2}{\e^3}\}}\le \psi_1(\tau).
\end{equation*}
Since $\underline{z}\in W^{1,2}_{p,{\rm loc}}
((0,T)\times(1/{\e^3},\infty))$, $1<p<\infty$ and
for sufficiently small $\e>0$, $\underline{z}$ satisfies that for all $(\tau,x)\in(0,T)\times(1/{\e^3},\infty)$,
\begin{align*}
    \pat_\tau \underline{z} -\mathcal{L}\underline{z} 
    &= \left[ (1-\e)(\e_0-\e_1)
    \mathbf{1}_{\{(1-\e)\mathcal{E}< \p_1\}}
    +(\p'_1 +r\p_1)
    \mathbf{1}_{\{(1-\e)\mathcal{E}\ge \p_1\}}
    \right]
    \\
    &\quad +\e^5\left[\t^2+(2\b-2r+\t^2)\left(x-\frac{2}{\e^3}\right)-r\left(x-\frac{2}{\e^3}\right)^2\right]\mathbf{1}_{\{x\le \frac{2}{\e^3}\}}
    \\
    & \le \left[ (1-\e)(\e_0-\e_1)
    \mathbf{1}_{\{(1-\e)\mathcal{E}< \p_1\}}
    +(\e_0-\e_1-\k)
    \mathbf{1}_{\{(1-\e)\mathcal{E}\ge \p_1\}}
    \right]
    \\
    &\quad+\e^5\left[\t^2-2r\left(x-\frac{2}{\e^3}\right)\right]\mathbf{1}_{\{x\le \frac{2}{\e^3}\}}
    \\
    & \le (\e_0-\e_1)+\e^5\left(\t^2+\frac{2r}{\e^3}\right)
    -\k'\e.
\end{align*}
Here, $\kappa>0$ is derived from the strict inequality in the Assumption~\ref{as2}(ii)
and since
$\varepsilon^5\!\left(\theta^2+r/{\varepsilon^3}\right)
= r\varepsilon^2 + O(\varepsilon^5)$,
for sufficiently small $\varepsilon>0$ we have
$\varepsilon^5\!\left(\theta^2+r/{\varepsilon^3}\right)
\le \frac{\kappa'}{2}\varepsilon$,
and hence
\[
(\varepsilon_0-\varepsilon_1)
+ \varepsilon^5\!\left(\theta^2+\frac{r}{\varepsilon^3}\right)
- \kappa' \varepsilon
\le (\varepsilon_0-\varepsilon_1)
- \frac{\kappa'}{2}\varepsilon
\]
for some constant $k'>0$.
This leads to the conclusion that $\partial_\tau \underline{z}-\mathcal{L}\underline{z}\le U$ in $(0,T)\times(1/{\e^3},\infty)$.

Then, for each $(\tau,x)\in (0,T)\times(1/{\e^3},\infty)$ with $\e>0$ small enough to satisfy $\max\limits_{\tau\in[0,T]}\Gamma_0(\tau)<1/{\e^3}$, it follows that $v(\tau,x)>-\p_0(\tau)$ (see lemma \ref{reg:s_1}).
Moreover, $\underline{z}$ satisfies
\begin{equation*}
    \begin{cases}
        \pat_\tau \underline{z}(\tau,x)
        -\mathcal{L}\underline{z}(\tau,x)
        \le U(x) \quad \text{for }\ (\tau,x)\in(0,T)\times(1/{\e^3},\infty),
        \\
        \underline{z}(\tau,1/{\e^3})\le 
        \p_1(\tau)-1/\e<-\p_0(\tau), \quad \tau\in[0,T]
        \\
        \underline{z}(0,x)\le 0 = v(0,x),
        \quad x\in[1/{\e^3},\infty).
    \end{cases}
\end{equation*}
Thus, by applying the comparison principle for the variational inequality \cite[Theorem 1.2]{YangYan2008} to $\underline{z}$ and $v$ in the domain $(0,T)\times(1/{\e^3},\infty)$, we get
\[\underline{z}\le v \quad \text{in }\ [0,T]\times[1/{\e^3},\infty).\]
Since $v\le \p_1$, this implies
\[v(\tau,x)=\p_1(\tau)
\quad \text{for all }\ (\tau,x)\in \{(1-\e)\mathcal{E}\ge \p_1\}\times[2/\e^3,\infty).\]
Let us fix some $c>0$. Recall that from the condition ii) in Assumption \ref{as2}, the function $e^{r\tau}(\mathcal{E}(\tau)-\p_1(\tau))$ is monotone increasing with respect to $\tau$ and $\mathcal{E}-\p_1\ge 0$ in $[\tau_1,T]$.
This implies that for all $\tau\in[\tau_1+c,T]$,
\begin{equation*}
    e^{rT}\big(\mathcal{E}(\tau)-\psi_1(\tau)\big)
    \\
    \ge e^{r\tau}\big(\mathcal{E}(\tau)-\psi_1(\tau)\big)
    \\
    \ge e^{r(\tau_1+c)}\big(\mathcal{E}(\tau_1+c)-\psi_1(\tau_1+c)\big).
\end{equation*}
Using this inequality, we deduce that for all $\tau\in [\tau_1+c,T]$ and $\e>0$,
\begin{align*}
    (1-\e)\mathcal{E}(\tau)-\p_1(\tau)
    &\ge \mathcal{E}(\tau)-\p_1(\tau)-\e\frac{\e_0-\e_1}{r}
    \\
    &\ge e^{r(\tau_1-T+c)}\big(\mathcal{E}(\tau_1+c)-\p_1(\tau_1+c)\big)-\e\frac{\e_0-\e_1}{r}.
\end{align*}
By taking $\e>0$ sufficiently small, we obtain that the right hand side of the above inequality is non-negative.
This yields that for sufficiently small $\e>0$ depending on $c>0$, the set $[\tau_1+c,T]$ is contained in
$\{\tau\in[0,T]:(1-\e)\mathcal{E}(\tau)\ge \psi_1(\tau)\}$.
As a result,
\[v(\tau,x)=\p_1(\tau)
\quad \text{for all }\ (\tau,x)\in[\tau_1+c,T]\times[2/\e^3,\infty).\]
This proves the second assertion with $K_c=2/\e^3$.
\end{proof}

The previous lemmas together with the continuity of $v$ deduce that
\begin{equation}\label{G_0}
    \chi_0(\tau):=\max\{x\in \mathbb{R}:v(\tau,x)=-\p_0(\tau)\},
\quad \text{for }\ \tau\in(0,T),
\end{equation}
and
\begin{equation}\label{G_1}
    \chi_1(\tau):=\min\{x\in \mathbb{R}:v(\tau,x)=\p_1(\tau)\},
    \quad \text{for }\ \tau\in(\tau_1,T)
\end{equation}
are well-defined. Moreover, if $\chi_0$ and $\chi_1$ are continuous, then the free boundaries $\pat\{v>-\p_0\}$ and $\pat\{v<\p_1\}$ can be parametrized by $\chi_0$ and $\chi_1$, respectively. In the next subsection, we will show that $\chi_0$ and $\chi_1$ are Lipschitz continuous and, furthermore, that $\chi_0$ and $\chi_1$ are smooth if $\psi_0$ and $\psi_1$ are smooth.

We end this subsection by investigating the limiting behavior of $\chi_0$ and $\chi_1$. 
\begin{lem}\label{fb:lim}
    Let $\chi_0$ and $\chi_1$ be defined as above. Then, 
    \begin{equation*}
        \displaystyle\lim_{\tau\to0^{+}}\chi_0(\tau)
        =-\infty
        \quad \text{and}\quad
        \displaystyle\lim_{\tau\to\tau_1^+}\chi_1(\tau)
        =\infty.
        \end{equation*}
\end{lem}
\begin{proof}
We argue by contradiction. Suppose that $\chi_0(\tau)$ does not diverge to $-\infty$
as $\tau\to0^+$. Then there exist a sequence $\tau_n\to0^+$ such that
$\chi_0(\tau_n)$ is bounded from below. Since $\chi_0(\tau)$ is bounded above by $\Gamma_0(\tau)$ on $[0,T]$,
we may extract a subsequence for which $\chi_0(\tau_n)$ converges.
By continuity of $v$ and $\p_0$, the set
$\{(\tau,x)\in [0,T]\times \mathbb R: v(\tau,x)=-\p_0(\tau)\}$ is closed, and hence its limit point
$(0,\lim_{n\to\infty} \chi_0(\tau_n))$ belongs to this set. This yields
$v(0,\cdot)=-\p_0(0)<0$ at such point, contradicting $v(0,\cdot)=0$.

The second limit follows by the same argument, using the result that $v(\tau_1,\cdot)\neq \p_1(\tau_1)$ in Lemma~\ref{reg:s_1}.
\end{proof}

\subsection{Estimate involving time and spatial derivative}
For the regularity of the free boundary, it is important to verify that the time and spatial derivatives 
of the solution remain comparable in the region $\{-\p_0 < v < \p_1\}$, 
even as the solution approaches the free boundary.
To obtain an estimate that reflects this relationship, we follow the method developed in \cite{ZJ2025}, in which a desired estimate holds throughout the entire domain.
However, unlike in \cite{ZJ2025}, our obstacle functions $\psi_0$ and $\psi_1$ depends on $\tau$, which blocks us from obtaining such global estimate.
To overcome this difficulty, we decompose the domain $\mathcal{D}_{T}$ 
into two separate regions and derive the desired estimates locally, in each part.
Here, we make essential use of the fact that 
\[
\{(\tau,x):\Gamma_0(\tau)<x<\Gamma_1(\tau),\ \tau\in(0,T)\}
\subset \{-\p_0<v<\p_1\},
\]
obtained in Lemma~\ref{reg:s_1}. 

\begin{lem}\label{estimate(3)}
Let $\Gamma_0$ and $\Gamma_1$ be the curves given in the previous subsection. 
Define each curves $\mathcal{M},\mathcal{M}_0:[0,T]\to\mathbb{R}$ by
\begin{equation}\label{def:M01}
\mathcal{M}(\tau)
=\frac{\Gamma_0(\tau)+\Gamma_1(\tau)}{2}
\quad \text{and} \quad
\mathcal{M}_0(\tau):=\frac{3\Gamma_0(\tau)+\Gamma_1(\tau)}{4}, \quad \tau\in(0,T).
\end{equation}
Then, there exists $\d>0$ and $\eta>0$
such that for all $\tau\in(0,T)$,
\begin{equation}\label{M01}\inf\limits_{(\tau-\d,\tau+\d)\cap [0,T]}\mathcal{M}-\sup\limits_{(\tau-\d,\tau+\d)\cap[0,T]}\mathcal{M}_0\ge \eta.\end{equation}
Furthermore, 
for each $(\tau,x)$ with $\tau\in(0,T)$  and $\chi_0(\tau)< x\le \mathcal{M}_0(\tau)$,
there exists a constant $C=C(\tau,\d,\eta)>0$ such that
the following estimate holds: 
\begin{equation}\label{eq:vderivatives}
    -\,C\,\partial_{x} v(\tau,x)
    \;\le\;
    \partial_{\tau} v(\tau,x) + \p_{0}'(\tau)
    \;\le\;
    C\,e^{\frac{1}{\gamma_{1}}x}\,\partial_{x} v(\tau,x).
\end{equation}
Note that if $I\Subset (0,T)$, we can choose the constant $C$ in the above inequality independent of the choice of $\tau\in I$.
\end{lem}
\begin{proof}
By the uniform continuity of $\mathcal{M}$ and $\mathcal{M}_0$, together with $\mathcal{M}_0<\mathcal{M}$, in $[0,T]$, one can find $\eta$ and $\delta$ such that   \eqref{M01} holds true. We now prove \eqref{eq:vderivatives}.

Fix any $\tau_0\in (0,T)$ and let $\d_0\in(0,\delta)$ sufficiently small to be determined. First, we choose $\d_0$ small so that $I:=(\tau_0-\d_0,\tau_0+\d_0)\subset (0,T)$, and 
\[\mathcal{R}:=\{(\tau,x)\in I\times\mathbb{R}:x\le \mathcal{M}(\tau_0)\}.\]
Recall the constant $X_c$ in the previous Lemma~\ref{reg:s_2} when $c=\tau_0-\d_0$. Then, we have $v=-\p_0$ in $[\tau_0-\d_0,T)\times(-\infty,-X_c]$.
For each large $n\in\mathbb{N}$ satisfying $-\frac{n}{4}<-X_c$ and $n>\max_{\tau\in [0,T]} \mathcal{M}(\tau)$, we set a penalized problem by
\begin{equation}\label{obs:pen,omega}
    \begin{cases}
        \pat_\tau \omega_{n,\e}-\mathcal{L}\omega_{n,\e}- \b_{0,\e}(\omega_{n,\e}+\p_0)=U
        \quad \text{in }\ \mathcal{R}^n,
        \\
        \omega_{n,\e}(\tau_0-\d_0,x)=v(\tau_0-\d,x)\ge - \psi_0(\tau_0-\d_0)
        \quad \text{for all }\ x\in[-n,\mathcal{M}(\tau_0-\d_0)],\\
        \omega_{n,\e}(\tau,-n)=v(\tau,-n)=-\p_0(\tau)\\
        \text{and}\quad 
        \omega_{n,\e}(\tau,\mathcal{M}(\tau))=v(\tau,\mathcal{M}(\tau))\ge -\psi_0(
        \tau)
        \quad \text{for all }\ \tau\in \bar{I}.
    \end{cases}
\end{equation}
where $\mathcal{R}^n:=\mathcal{R}\cap \mathcal{D}_T^n$. Note that, since $\mathcal{M}$ is smooth and $\p_0$ is of $C^{1,1}$, the lateral boundary of $\mathcal{M}$ is smooth and the prescribed lateral boundary data $-\p_0(\cdot)$ and $v(\cdot,\mathcal{M}(\cdot))$ belong to $C^{1,1}((\tau_0-\delta_0,\tau_0+\delta_0))$, and that the initial datum $v(\tau_0-\d_0,\cdot)$ lies in $W^{2,p}((-n,\mathcal{M}(\tau_0-\d_0)))$ for all $1<p<\infty$. Therefore, the existence of $\omega_{n,\varepsilon}$ in $W^{1,2}_{p}$ follows from the Schauder fixed point theorem \cite[Theorem 8.3]{Lieberman}.
Moreover, from the fact that $\b_{0,\e}\ge 0$, \eqref{U} and \eqref{psi12}, 
\[\pat_\tau(- \p_{0})-\mathcal{L}(-\p_0)- \b_{0,\e}(-\p_0+\p_0)\le \pat_\tau(- \p_{0})-r \p_{0} \le \varepsilon_0-\varepsilon_1 \le U
\quad \text{in }\ \mathcal{R}^n,\]
Therefore, by the comparison principle \cite[Theorem 1.1]{YangYan2008}, we have $\o_{n,\e}\ge-\p_0$ in $\mathcal{R}^n$.

By using the $W^{1,2}_p$-estimate in \cite[Theorem 7.30 and Corollary 7.31]{Lieberman} and proceeding as in the proof of Lemma~\ref{lem:v_existence}, we obtain the existence and uniqueness of a solution $\omega\in W^{1,2}_{p,{\rm loc}}(\mathcal{R})$ to the
corresponding obstacle problem on $\mathcal{R}$:
\begin{equation*}
    \begin{cases}
        \pat_\tau \omega -\mathcal{L}\omega=U
        \quad \text{for }\ \omega(\tau,x)>-\p_0(\tau),
        \\
        \pat_\tau \omega-\mathcal{L}\omega\ge U
        \quad \text{for }\  \omega(\tau,x)=-\p_0(\tau),
        \\
        \omega(\tau_0-\d_0,x)=v(\tau_0-\d_0,x)
        \quad \text{for all }\ x\in(-\infty,\mathcal{M}(\tau_0-\d)],
        \\
        \omega(\tau,\mathcal{M}(\tau))=v(\tau,\mathcal{M}(\tau))
        \quad \text{for all }\ \tau\in\bar{I}.
    \end{cases}
\end{equation*}
Note that, by the uniqueness,
$\omega=v$ in $\overline{\mathcal{R}}$
and, by the weak convergence, all uniform estimate for $\o_{n,\e}$ carries over directly to $v$.
Moreover, by following the same procedure as in the Lemma~\ref{estimate(2)}, $\pat_x \o_{n,\e}\ge 0$ in $\overline{\mathcal{R}}$ and $\omega_{n,\e}+\p_0$ also satisfies the following estimate in $\overline{\mathcal{R}^n}$:
\begin{equation}\label{omega_tau}
    -\tfrac{\g_1}{1-\g_1}(L_1^{\gam}-L_0^{\gam})e^{K_1(T-\tau)-\frac{1}{\g_1}x} -2q
    \le \pat_\tau \omega_{n,\e}+ \p'_0
    \le (\e_0 - \e_1)e^{-r\tau}+2q.
\end{equation}
Equipped with the above uniform estimates, we now proceed to derive the desired inequality at any point $(\tau_0,x_0)$ with  $\chi_0(\tau_0)< x_0 \le \mathcal{M}_0(\tau_0)$.
We take into account an auxiliary function $\Phi(\tau, x)=\Phi(\tau,x;x_0)$ for $(\tau,x)\in \overline{\mathcal{R}^n}$ given by
\[\Phi(\tau,x)
=e^{k\tau}\left( e^{\frac{2}{\g_1}\vert x-x_0 \vert}-\frac{2}{\g_1}\vert x-x_0 \vert-1 \right),\]
where $k=2(\frac{2}{\g_1^2}+\frac{1}{\g_1})\t^2+\frac{4}{\g_1}(\b+r)-r+4q+2$.
It is straightforward to check that $\Phi \in W^{1,2}_{p,{\rm loc}}(\mathcal{R}^n) \cap C(\overline{\mathcal{R}^n})$ for all $1<p<\infty$
    and hence with $\mathcal{L}_\b:=\mathcal{L}+\b'_{0,\e}(\cdots)$, direct calculation yields
    \begin{align*}
        \pat_\tau\Phi - \mathcal{L}_\b \Phi\ge
        \pat_\tau\Phi - \mathcal{L} \Phi
        & = e^{k \tau} \Big[
        (k + r)\Big(  e^{\frac{2}{\g_1}|x - x_0|} - \frac{2}{\g_1}|x - x_0| - 1 \Big)
        - \frac{2\t^2}{\g_1^2}
         e^{\frac{2}{\g_1}|x - x_0|} 
        \\
        & \quad 
        - \frac{2}{\g_1}\Big(\b - r + \frac{\t^2}{2}\Big) \operatorname{sign}(x - x_0)     
        \big(  e^{\frac{2}{\g_1}|x - x_0|} - 1 \big)
        \Big]
        \\
        &\ge e^{k\tau}\Big[
        \frac{k + r}{2}\Big(  e^{\frac{2}{\g_1}|x - x_0|} - \frac{4}{\g_1}|x - x_0| - 2\Big) 
        \\
        &\quad +\Big\{ \frac{k+r}{2}-\frac{2}{\g_1}(\b+r)-\Big(\frac{2}{\g_1^2}+\frac{1}{\g_1}\Big)\t^2 \Big\}  e^{\frac{2}{\g_1}|x - x_0|}\Big]
        \\
        &\ge e^{k\tau}\Big[
        -(k+r)\mathbf{1}_{\{|x - x_0|\le \g_1\}}
        + (2q+1)e^{\frac{2}{\g_1}|x - x_0|}\Big],
    \end{align*}
    where the first inequality follows from $\b'_{0,\e}\le 0$ and $\Phi\ge 0$.
Furthermore, let 
\[C_1
:=\Bigg[\frac{C_2(1-\g_1)(k+r)}
{L_1^{\frac{\g_1-\g}{\g_1}}
-L_0^{\frac{\g_1-\g}{\g_1}}}+\g_1\Bigg]e^{1+(\vert k\vert +K_1)T},\]
where $C_2>0$ is a large constant to be chosen later.
We then define $\Psi(\tau,x)=\Psi(\tau,x;x_0)$ for $(\tau,x)\in\overline{\mathcal{R}^n}$ by
\begin{align*}
    \Psi(\tau,x)&:=(\tau-\tau_0+\d_0)\left\{\pat_\tau \omega_{n,\e}(\tau,x)+\p'_0(\tau)\right\}
    \\
    &\qquad +C_1\pat_x \omega_{n,\e}(\tau,x)
    +C_2e^{-\frac{1}{\g_1}\max\{x_0,0\}}\Phi(\tau,x).
\end{align*}
Observe that by differentiating the equation \eqref{obs:pen,omega} with respect to $\tau$ and $x$, we see that
\[\partial_\tau (\pat_\tau\o_{n,\e}+\p'_0)-\mathcal{L}_\b(\pat_\tau\o_{n,\e}+\p'_0)=\p''_0+r\p'_0
\quad \text{and} \quad
\partial_\tau (\pat_x\o_{n,\e})-\mathcal{L}_\b(\pat_x\o_{n,\e})=U'(x)\]
in $\mathcal{R}_n$.
We then use the estimate \eqref{omega_tau} and see that $\Psi$ satisfies
\begin{align*}
    &\pat_\tau\Psi(\tau,x) -\mathcal{L}_\b\Psi(\tau,x)
    \\
    &\ge \pat_\tau \omega_{n,\e}(\tau,x)+\p'_0(\tau)
    +(\tau-\tau_0+\d_0)
    \{\p''_0(\tau)+r\p'_0(\tau)\}+C_1U'(x)
    \\
    &\quad +C_2e^{k\tau-\frac{1}{\g_1}\max\{x_0,0\}}
    \Big[
    -(k+r)\mathbf{1}_{\{|x - x_0|\le \g_1\}}
    + (2q+1)e^{\frac{2}{\g_1}|x - x_0|}\Big]
    \\
    &\ge -\tfrac{\g_1}{1-\g_1}(L_1^{\gam}-L_0^{\gam})e^{K_1(T-\tau)-\frac{1}{\g_1}x} 
    -2\{1+\d_0(r+1)\}q
    \\
    &\quad +C_1U'(x)+C_2e^{k\tau-\frac{1}{\g_1}\max\{x_0,0\}}
    \Big[
    -(k+r)\mathbf{1}_{\{|x - x_0|\le \g_1\}}+(2q+1)\Big]
\end{align*}
for all $(\tau,x)\in\mathcal{R}^n$.
Choosing $C_2>1$ large and $\d_0>0$ sufficiently small enough to satisfy $(2q+1)C_2e^{k\tau-\frac{1}{\g_1}\max\{x_0,0\}}>2q+1>2\{1+\d_0(r+1)\}q$, we deduce that
\begin{align*}
    \pat_\tau\Psi(\tau,x) -\mathcal{L}_\b\Psi(\tau,x)
    &\ge -\tfrac{\g_1}{1-\g_1}(L_1^{\gam}-L_0^{\gam})e^{K_1(T-\tau)-\frac{1}{\g_1}x} +C_1U'(x)
    \\
    &\quad +C_2e^{k\tau-\frac{1}{\g_1}\max\{x_0,0\}}
    \Big[
    -(k+r)\mathbf{1}_{\{|x - x_0|\le \g_1\}}\Big].
\end{align*}
Moreover, note
\[C_1U'(x)=\tfrac{\g_1}{1-\g_1}(L_1^{\gam}-L_0^{\gam})e^{1-\frac{1}{\g_1}x+(\vert k\vert +K_1)T}+C_2(k+r)e^{1-\frac{1}{\g_1}x+(\vert k\vert +K_1)T}\]
and 
$x-\max\{x_0,0\}\le x-x_0 \le \g_1$ in $\{\vert x-x_0 \vert \le \g_1\}$, implies that
\[e^{-\frac{1}{\g_1}\max\{x_0,0\}} \le e^{-\frac{1}{\g_1}x} 
\quad \text{in }\ \{\vert x-x_0 \vert \le \g_1\}.\]
This leads to the conclusion that $\partial_\tau \Psi - \mathcal{L}_\b\Psi \ge 0$ in $\mathcal{R}^n$. 

On the right lateral boundary of $\mathcal{R}^n$, we recall $\pat_x \omega_{n,\e}(\tau,\mathcal{M}(\tau))\ge 0$ for all $\tau\in I$ and the estimate \eqref{omega_tau} to see
\begin{align*}
    \Psi(\tau,\mathcal{M}(\tau))&=
    (\tau-\tau_0+\d_0) \left\{\pat_\tau \omega_{n,\e}(\tau,\mathcal{M}(\tau))
    +\p'_0(\tau)\right\}+C_1\pat_x \omega_{n,\e}(\tau,\mathcal{M}(\tau))
    \\
    &\qquad+C_2e^{-\frac{1}{\g_1}\max\{x_0,0\}}\Phi(\tau,\mathcal{M}(\tau))
    \\
    &\ge -\tfrac{2\d \g_1}{1-\g_1}(L_1^{\gam}-L_0^{\gam})e^{K_1(T-\tau)-\frac{\mathcal{M}(\tau)}{\g_1}}-4\d q
    \\
    &\qquad +C_2e^{k\tau-\frac{1}{\g_1}\max\{x_0,0\}}\Big( e^{\frac{2}{\g_1}\vert \mathcal{M}(\tau)-x_0 \vert}-\tfrac{2}{\g_1}\vert \mathcal{M}(\tau)-x_0 \vert-1 \Big).
\end{align*}
We note from \eqref{M01} that for all $\tau\in I$,
\begin{equation}\label{M1,M2:right}
    \mathcal{M}(\tau)-x_0
    \ge \mathcal{M}(\tau)-\mathcal{M}_0(\tau_0)
    \ge \inf_{\tau\in I}\mathcal{M}(\tau)-\sup_{\tau\in I}\mathcal{M}_0(\tau)
    \ge \eta.
\end{equation}
Therefore, since $x_0$ is bounded above by $\mathcal{M}$, and the function $\mathcal{M}$ is bounded, we choose $C_2$ sufficiently large to get $\Psi(\tau,\mathcal{M}(\tau))\ge 0$ for all $\tau\in I$.
On the left lateral boundary of $\mathcal{R}^n$,
since $\o_{n,\e}(\tau,-n)=-\p_0(\tau)$ and
$\omega_{n,\e}(\tau,x)\ge -\p_0(\tau)$ holds for all $(\tau,x)\in\mathcal{R}^n$, we have $\pat_x \omega_{n,\e}(\tau,-n)\ge 0$ for all $\tau\in I$. 
This yields that
\begin{align*}
    \Psi(\tau,-n)
    &\ge -\tfrac{2\d\g_1}{1-\g_1}
        (L_1^{\gam}-L_0^{\gam})e^{K_1(T-\tau)+\frac{n}{\g_1}}-4\d q
        \\
        &\qquad +C_2e^{-\frac{1}{\g_1}\max\{x_0,0\}}\Big( e^{\frac{2}{\g_1}\vert n+x_0\vert}-\tfrac{2}{\g_1}\vert n+x_0\vert-1 \Big).
\end{align*}
Moreover, since the assumption $-\frac{n}{4}<-X_c\le x_0$ implies $n+x_0>\frac{3}{4}n$, for sufficiently large $n\in\mathbb{N}$ we have $\Psi(\tau,-n)\ge 0$ for all $\tau\in I$.
Finally, on the initial boundary $\{0\}\times[-n,\mathcal{M}(\tau_0-\d_0)]$ of $\mathcal{R}^n$,
since $\pat_x \omega_{n,\e}(\tau_0-\d_0,x)=\pat_x v(\tau_0-\d_0,x)\ge 0$, we deduce that
$\Psi(\tau_0-\d,x)=C_1\pat_x \omega_{n,\e}(\tau_0-\d_0,x)+e^{-\frac{1}{\g_1}\max\{x_0,0\}}C_2\Phi(\tau_0-\d_0,x)\ge 0$.

Therefore, by applying the comparison principle \cite[Theorem 1.1]{YangYan2008}, we conclude that
for all $(\tau,x)\in \mathcal{R}^n$,
\begin{align*}
    \Psi(\tau,x)&=(\tau-\tau_0+\d_0)\left\{\pat_\tau \omega_{n,\e}(\tau,x)+\p'_0(\tau)\right\}
    \\
    &\qquad +C_1\pat_x \omega_{n,\e}(\tau,x)+C_2e^{-\frac{1}{\g_1}\max\{x_0,0\}}\Phi(\tau,x) \ge 0.
\end{align*}
By the weak convergence of $\omega_{n,\varepsilon}$ to $v$ in $W^{1,2}_{p,{\rm loc}}(\mathcal{R})$, we obtain that for almost every $(\tau,x)\in \mathcal{R}$,
\begin{align*}
    (\tau-\tau_{0}+\delta_0)\bigl\{\partial_{\tau}v(\tau,x)+\p_{0}'(\tau)\bigr\}
    + C_{1}\partial_{x}v(\tau,x)
    + C_{2}e^{-\frac{1}{\gamma_{1}}\max\{x_{0},0\}}\Phi(\tau,x)  \ge 0.
\end{align*}
Since $(\tau_{0},x_{0})\in \mathcal{D}_T \cap \{-\p_{0}<v<\p_{1}\}$ and $v$ is smooth in this set, putting $(\tau,x)=(\tau_0,x_0)$ in the previous inequality, we conclude that
\[
\delta_0 \bigl\{\partial_{\tau}v(\tau_{0},x_{0})+\p_{0}'(\tau_{0})\bigr\}
+ C_{1}\partial_{x}v(\tau_{0},x_{0})
\ge 0.
\]
Thus, we prove the first inequality in \eqref{eq:vderivatives}, with $(\tau,x)=(\tau_0,x_0)$, by choosing $C\ge \frac{C_1}{\delta_0}$. Note that $\delta_0$ depends on $\tau_0$.

On the other hand, to derive the second inequality  in \eqref{eq:vderivatives},
we introduce the auxiliary function $\widetilde{\Psi}(\tau,x)=\widetilde{\Psi}(\tau,x;x_0)$ for $(\tau,x)\in\mathcal{R}^n$ defined by
\[\widetilde{\Psi}(\tau,x):=(\tau-\tau_0+\d_0)\left\{\pat_\tau \omega_{n,\e}(\tau,x)+\p'_0(\tau)\right\}-C_3e^{\frac{1}{\g_1}x_0}\pat_x \omega_{n,\e}(\tau,x)-C_4\Phi(\tau,x),\]
where $C_3$ is a constant such that
\[C_3 = \displaystyle\frac{C_4(1-\g_1)(k+r)}
    {L_1^{\gam}-L_0^{\gam}}e^{1+\vert k\vert T},\]
and $C_4>0$ is a large constant to be chosen later.
Then, using the estimate \eqref{omega_tau}, we see that $\widetilde{\Psi}$ satisfies
\begin{align*}
    \pat_\tau\widetilde{\Psi}(\tau,x) -\mathcal{L}_\b\widetilde{\Psi}(\tau,x)
    &\le (\e_0-\e_1)e^{-r\tau}+2\{1+\d_0(r+1)\}q-C_3e^{\frac{1}{\g_1}x_0}U'(x)
    \\
    &\quad -C_4e^{k\tau}\Big[
        -(k+r)\mathbf{1}_{\{|x - x_0|\le \g_1\}}
        + 3qe^{\frac{2}{\g_1}|x - x_0|}\Big]
    \\
    &\le (\e_0-\e_1)e^{-r\tau}+2\{1+\d_0(r+1)\}q-3qC_4
    \\
    &\quad -C_4(k+r)e^{\vert k\vert T+1-\frac{x-x_0}{\g_1}}
    +C_4(k+r)e^{k\tau}\mathbf{1}_{\{|x - x_0|\le \g_1\}}.
\end{align*}
Observe that the term  
$(\e_0-\e_1)e^{-r\tau}$ is bounded above by $\e_0-\e_1$.
Thus, by choosing $C_4>1$ large
enough to satisfy $\e_0-\e_1<\frac{1}{2}qC_4$
and $\d_0(r+1)<\frac{1}{4}C_4$, we deduce that
$\pat_\tau\widetilde{\Psi} - \mathcal{L}_\b\widetilde{\Psi}\le 0$
in $\mathcal{R}^n$.

On the right lateral boundary of $\mathcal{R}^n$, we recall the estimate $\pat_x \omega_{n,\e}(\tau,\mathcal{M}(\tau))\ge 0$ and the estimate \eqref{omega_tau} to see
\begin{align*}
    \widetilde{\Psi}(\tau,\mathcal{M}(\tau))&=(\tau-\tau_0+\d_0) \left\{\pat_\tau \omega_{n,\e}(\tau,\mathcal{M}(\tau))
    +\p'_0(\tau)\right\}
    \\
    &\qquad -C_3e^{\frac{1}{\g_1}x_0}\pat_x \omega_{n,\e}(\tau,\mathcal{M}(\tau))-C_4\Phi(\tau,\mathcal{M}(\tau))
    \\
    &\le 2\d(\e_0-\e_1)+4\d q
    -C_4\left( e^{\frac{2}{\g_1}\vert \mathcal{M}(\tau)-x_0 \vert}-\frac{2}{\g_1}\vert \mathcal{M}(\tau)-x_0 \vert-1 \right).
\end{align*}
Applying \eqref{M1,M2:right} and taking $C_4$
 sufficiently large yields $\widetilde{\Psi}(\tau,\mathcal{M}(\tau))\le 0$ for all $\tau\in I$.
On the left lateral boundary of $\mathcal{R}^n$,
since $\o_{n,\e}(\tau,-n)=-\p_0(\tau)$ and
$\omega_{n,\e}\ge -\p_0$ holds in $\mathcal{R}^n$, we have $\pat_x \omega_{n,\e}(\tau,-n)\ge 0$ for all $\tau\in[\tau_0-\d,\tau_0+\d]$.
Therefore, for large $C_4>0$,
\begin{align*}
    \widetilde{\Psi}(\tau,-n)&=(\tau-\tau_0+\d_0) \left\{\pat_\tau \omega_{n,\e}(\tau,-n)+\p'_0(\tau)\right\}
    \\
    &\qquad -C_3e^{\frac{1}{\g_1}x_0}\pat_x \omega_{n,\e}(\tau,-n)-C_4\Phi(\tau,-n)
    \\
    &\le -C_4\left( e^{\frac{2}{\g_1}\vert n+x_0\vert}-\frac{2}{\g_1}\vert n+x_0\vert-1 \right).
\end{align*}
Therefore, we deduce $\widetilde{\Psi}(\tau,-n)\le 0$ for all $\tau\in I$.
Finally, on the initial boundary $\{0\}\times[-n,\mathcal{M}(\tau_0-\d)]$ of $\mathcal{R}^n$,
since $\pat_x \omega_{n,\e}(\tau_0-\d_0,x)=\pat_x v(\tau_0-\d_0,x)\ge 0$, we deduce that
$\widetilde{\Psi}(\tau_0-\d,x)=-C_3e^{\frac{1}{\g_1}x_0}\pat_x \omega_{n,\e}(\tau,x)
-C_4\Phi(\tau,x)\le 0$. 
Therefore, by applying the comparison principle \cite[Theorem 1.1]{YangYan2008}, we conclude that
$\widetilde{\Psi} \le 0$ in $\mathcal{R}^n$. 
This implies that, passing $\varepsilon\to 0^{+}$ and $n\to\infty$, for almost every $(\tau,x)\in \mathcal{R}$,
\begin{align*}
   (\tau-\tau_0+\d_0)\left\{\pat_\tau v(\tau,x)+\p'_0(\tau)\right\}
    -C_3e^{\frac{1}{\g_1}x_0}\pat_x v(\tau,x)+C_4\Phi(\tau,x) \le 0.
\end{align*}
In particular, by putting $(\tau,x)=(\tau_0,x_0)$, the continuity of $\pat_\tau v$ and $\pat_x v$ yields
\[
\d_0 \left(\pat_\tau v(\tau_0,x_0)+\p'_0(\tau_0)\right)
    -C_3e^{\frac{1}{\g_1}x_0}\pat_x v(\tau_0,x_0)\le 0.
\]
Therefore, we prove the second inequality in \eqref{eq:vderivatives}, with $(\tau,x)=(\tau_0,x_0)$, by choosing $C\ge \frac{C_3}{\delta_0}$.
\end{proof}

To derive the estimate near the free boundary $\partial\{v<\p_1\}$,
we formulate an analogous penalized problem in the domain
$(\tau_1,T)\times(\mathcal{M}(\cdot),n)$ and proceed as in the previous lemma.
The only difference is that the boundary condition on the right lateral boundary is prescribed by $\p_1$ and we use the fact that $x_0$ is bounded above by the constant $K_c>0$ with 
$c\le \tau_0-\d$ in the Lemma~\ref{reg:s_2}.
As a consequence, we obtain the following result.

\begin{lem}\label{estimate(4)}
Let $\Gamma_0$ and $\Gamma_1$ be the curves given in the previous subsection and $\mathcal{M}$ be the curve given in the previous lemma.
Define by \[\mathcal{M}_1(\tau):=\frac{\Gamma_0(\tau)+3\Gamma_1(\tau)}{4}
\quad \text{for all }\ \tau\in(0,T).\]
Then, for each $(\tau,x)$ with $\tau\in(\tau_1,T)$ and $\mathcal{M}_1(\tau)\le x < \chi_1(\tau)$,
there exists a constant $C=C(\tau,\d,\eta,)>0$, where $\d$ and $\eta$ are the constants given in the previous lemma,  such that
the following estimate holds:
\[
    -\,C\,\partial_{x} v(\tau,x)
    \;\le\;
    \partial_{\tau} v(\tau,x) - \p_{1}'(\tau)
    \;\le\;
    C\,e^{\frac{1}{\gamma_{1}}x}\,\partial_{x} v(\tau,x).
\]
\end{lem}

\subsection{Regularity of the free boundary}
In this subsection, we establish that the free boundaries $\partial \{v > -\psi_0\}$ and $\partial \{v < \psi_1\}$ are Lipschitz continuous under the basic assumption that the obstacle functions satisfy $\psi_0, \psi_1 \in C^{1,1}([0,T])$. 
We emphasize that, throughout the previous sections, these regularity assumptions were sufficient to derive the comparison principles and the estimates required for the analysis of the obstacle problem. 
Furthermore, we show that the free boundaries are smooth provided that the obstacle functions are smooth.

We first show that the free boundaries $\chi_0$ and $\chi_1$ are locally Lipschitz continuous.
This step relies only on the $C^{1,1}$ regularity of the obstacles and the estimates obtained
earlier.
After establishing the local Lipschitz regularity of each free boundary, we apply the boundary Harnack inequality in neighborhoods of free boundary points inside the Lipschitz domains $\{v>-\psi_0\}$ and $\{v<\psi_1\}$. As a consequence, the regularity of each free boundary can be further improved, leading eventually to their smoothness.

In this subsection, we establish that the free boundaries $\partial \{v > -\psi_0\}$ and $\partial \{v < \psi_1\}$ are Lipschitz continuous, provided that the obstacle functions $\psi_0, \psi_1$ satisfy the basic $C^{1,1}([0,T])$ regularity. 
We first show that the parametrizations $\chi_0$ and $\chi_1$ of the free boundaries are locally Lipschitz continuous, a result that relies on the $C^{1,1}$ regularity of the obstacles and the estimates derived in the preceding sections. 
By applying the boundary Harnack inequality in neighborhoods of the free boundary points within the Lipschitz domains $\{v > -\psi_0\}$ and $\{v < \psi_1\}$, we first improve the regularity of the free boundaries beyond the Lipschitz class. 
Building upon this improved regularity, we further demonstrate that the free boundaries are smooth provided that the obstacle functions themselves are smooth.

\begin{thm}[Lipschitz continuity]
    The parametrizations $\chi_0$ and $\chi_1$  of the free boundaries $\partial \{v > -\psi_0\}$ and $\partial \{v < \psi_1\}$ of the obstacle problem \eqref{obs}, as described in \eqref{G_0} and \eqref{G_1}, are locally Lipschitz continuous.
\end{thm}
\begin{proof}
   We first prove the local Lipschitz continuity of $\chi_0$. 
    Let $I' \Subset I\Subset (0,T)$ be any open intervals, and choose $N\in \mathbb{N}$ large so that 
    \begin{equation}\label{lv,n}
        \frac{1}{N}\le \min\limits_{\tau\in I}\left\{v(\tau,\mathcal{M}_0(\tau))+\p_0(\tau)\right\},
    \end{equation}
where $\mathcal{M}_0$ is from \eqref{def:M01}.  For each $n \in \mathbb{N}$ with $n \ge N$, we consider the following level set:
\[ \ell_n := \{ (\tau, x) : \tau\in I,  \ v(\tau, x) + \psi_0(\tau) = \tfrac{1}{n} \}. \]
Since $\partial_x v > 0$ in the region $\{ -\psi_0 < v < \psi_1 \}$, the implicit function theorem implies that the parametrization of $\ell_n$  for $\tau\in I'$, denoted by $x_n:I'\to \mathbb{R}$, is well-defined and continuously differentiable. That is, $\ell_n = \{ (\tau, x_n(\tau)) : \tau \in I' \}$.   Note that $\chi_0(\tau)<x_n(\tau)$ for all $\tau\in I'$

  Clearly,    $-X_c\le x_n(\tau)\le \max_{\tau\in [0,T]}
\mathcal{M}_0(\tau)$ for all $\tau\in I'$, where $X_c$ is the constant in the Lemma~\ref{reg:s_2} with $c=\min_{\tau\in \overline{I'}}\tau$. 
    This implies that $\{x_n\}_{n=N}^\infty$ is uniformly bounded on $I'$.
    Furthermore, differentiating the both sides of the equation
    \begin{equation}\label{v,x_n}
        v(\tau,x_n(\tau))+\p_0(\tau)=\frac{1}{n}
    \end{equation}
     with respect to $\tau$, we get
    \begin{equation}\label{eq:x_n'}
        \pat_\tau v(\tau,x_n(\tau))+\p_0'(\tau)
        +\pat_x v(\tau, x_n(\tau))  x_n'(\tau)=0
        \quad \text{for all }\ \tau\in I'.
    \end{equation}
Then,    utilizing the the second inequality in \eqref{eq:vderivatives}, it follows that for all $n\ge N$,
    \begin{equation}\label{pat_tau x_n}
        \big\vert  x_n' (\tau)\big\vert 
        = \frac
        {\left\vert \pat_\tau v(\tau,x_n(\tau))+ \p_0'(\tau) \right\vert}
        {\left\vert \pat_x v(\tau,x_n(\tau))\right\vert}
        \le C\exp\left(\frac{\max\limits_{\tau\in[0,T]}\mathcal{M}_0(\tau)}{\g_1}\right)=: C'
        \quad \forall \tau\in I'.
    \end{equation}
    The above results yield that the sequence $\{x_n\}_{n=N}^\infty$ is equicontinuous in $I'$. Therefore, by the Arzelà–Ascoli theorem, there exists a subsequence $\{x_{n_i}\}_{i=1}^\infty$ that converges to a function $x:I'\to \mathbb{R}$ uniformly in $I'$.
   In inequality  \eqref{v,x_n}, replacing $x_n(\tau)$ with $x_{n_i}(\tau)$ and taking $i\to\infty$ yields $x(\tau)=\chi_0(\tau)$ for all $\tau\in I'$.
    Moreover, for every $\tau_1,\tau_2\in I'$ with $\tau_1<\tau_2$ , it follows from  the mean value theorem and \eqref{pat_tau x_n}  that
    \[\frac{\vert x_{n_i}(\tau_2)-x_{n_i}(\tau_1) \vert}{\tau_2-\tau_1}\le C'\] Then, taking $i\to\infty$ establishes
    the Lipschitz regularity of $\chi_0$ in $I'$. 
        
    The local Lipschitz regularity of $\chi_1$ can be shown analogously. We establish a sequence of level sets $\{(\tau,x):\tau\in (\tau_1,T),\ v(\tau,x)-\p_1(\tau)=-\frac{1}{n}\}$ for sufficiently large $n\in\mathbb{N}$, consider arbitrary open intervals $I'\Subset I\Subset (\tau_1,T)$, and follow the same procedure as in above. Note that for the uniform boundedness, we use the constant $K_c>0$, instead of $X_c$, with $c=\min_{\tau\in \overline{I'}} \tau$ in Lemma~\ref{reg:s_2} .
\end{proof}
\begin{thm}[Smoothness]
Let  $\chi_0$ and $\chi_1$  be the parametrization of the free boundaries $\partial \{v > -\psi_0\}$ and $\partial \{v < \psi_1\}$ of the obstacle problem \eqref{obs}, as described in \eqref{G_0} and \eqref{G_1}. If $\psi_0,\psi_1\in C^\infty((0,T))$, then
$\chi_0\in C^{\infty}((0,T))$ and $\chi_1\in C^{\infty}((\tau_1,T))$.
\end{thm}
\begin{proof}
        Let us first prove that $\chi_0$ is smooth. Fix a free boundary point $(\tau_0,\chi_0(\tau_0))$. For simplicity, we write $I_\varrho:=(t_0-\varrho,t_0+\varrho)$ and $Q_\varrho := I_{\varrho^2} \times (\chi_0(\tau_0)-\varrho,\chi_0(\tau_0)+\varrho)$ for $\varrho>0$.  Choose $r\in(0,1)$ sufficiently small so that $Q_r \Subset \{(\tau,x): \tau\in(0,T),\ x\in(-\infty,\mathcal{M}_0(\tau))\}$. 
We recall the parametrizations of the level curves $x_n=: \chi_{0,n}$, $n\ge N$, defined  in the proof of the previous theorem with $I'=I_{r^2}$.   Note that
    $v(\tau,\chi_{0,n}(\tau))+\p_0(\tau)=\frac{1}{n}$  and $(\tau,\chi_{0,n}(\tau))\in \{v>-\p_0\}\cap Q_r$ for all $\tau\in I_{r^2}$ and for all sufficiently large $n \ge N$. Without loss of generality, we may assume that $\chi_{0,n}$ converges to $\chi_0$ uniformly in $I_{r^2}$
   Recall that in $\{v>-\p_0\}\cap Q_r$, $v$ is smooth and satisfies the equation $\pat_\tau v -\mathcal{L} v =U$. By differentiating the previous equation equation with respect to $\tau$ and $x$, we deduce that
    $\pat_\tau v+\p'_0$ and $\pat_x v$ satisfies the following equations respectively:
    \begin{equation*}
        \begin{cases}
            \pat_\tau (\pat_\tau v+\p'_0) -\mathcal{L}  (\pat_\tau v+\p'_0) =\p''_0+r\p'_0
            \quad \text{in }\ \{v>-\p_0\}\cap Q_r,
            \\
            \pat_\tau v+\p'_0=0
            \quad \text{on }\ \pat\{v>-\p_0\},
        \end{cases}
    \end{equation*}
    \begin{equation*}
        \begin{cases}
            \pat_\tau (\pat_x v) -\mathcal{L} (\pat_x v) =U'
            \quad \text{in }\ \{v>-\p_0\}\cap Q_r,
            \\
            \pat_x v =0
            \quad \text{on }\ \pat\{v>-\p_0\},
        \end{cases}
    \end{equation*}
    where the boundary condition $\pat_\tau v = -\p'_0$ (in the trace sense) follows from the continuity of $\partial_\tau v$ at almost every free boundary point (see \cite[Theorem 1.1]{Blanchet-et-al-2006}).
    Since $\{v>-\p_0\}$ is a Lipschitz domain, we apply the boundary Harnack inequality \cite[Theorem 1.2]{TorresLatorre2024} to deduce that
    there exists $\d\in(0,1)$ and a constant $C$ that depends on $\Vert \pat_\tau v+\p'_0 \Vert_{L^\infty(Q_r)}$ and $\Vert \pat_x v \Vert_{L^\infty(Q_r)}$ such that
    \[
    \left\Vert \frac{\pat_\tau v+\p'_0}{\pat_x v} \right\Vert_{C^{\frac{\d}{2},\d}(\{v>-\p_0\}\cap Q_{r/2})}\le C,
    \]
    hence, by \eqref{eq:x_n'} and the uniform $C^{1,1}$ estimate of $\chi_{0,n}$ in \eqref{pat_tau x_n}, 
    \[
        \Vert \chi'_{0,n}(\tau) \Vert_{C^{\frac{\d}{2}}(I_{r^2/4})}
    =\left\Vert \frac{\pat_\tau v(\tau,\chi_{0,n}(\tau))+\p'_0(\tau)}{\pat_x v(\tau,\chi_{0,n}(\tau))} \right\Vert_{C^{\frac{\d}{2}}(I_{r^2/4})} \le C.
    \]
    Then, there exists a subsequence $\{\chi_{0,{n_j}}(\tau)\}_{j=1}^\infty$ such that $\chi_{0,{n_j}}'$ converges to a function $\mathcal{Y}$ in  $C^{\frac{\d}{2}}(I_{r^2/4})$.
 Therefore, we deduce that $\chi_0$ is differentiable and, in particular,
    $\chi'_0(\tau)=\displaystyle\lim_{k\to\infty}\chi'_{0,{n_k}}(\tau)=\mathcal{Y}(\tau)$ in $I_{r^2/4}$.
    Therefore, $\chi_0$ is a $C^{1+\frac{\d}{2}}$-function in $I_{r^2/4}$.
    
Since $\chi_0\in C^{1+\frac{\delta}{2}}(I_{r^2/4})$, the region
$\{v>-\psi_0\}\cap Q_{r/2}$ can be locally represented as a
parabolic domain whose boundary is of class
$C^{1+\frac{\delta}{2}}$.
We may therefore apply the higher order boundary Harnack inequality
\cite[Theorem 1.2]{Kukuljan2022} to conclude that the
$C^{1+\frac{\delta}{2}}$ norm of $\chi'_{0,n}$ is uniformly bounded
in $I_{r^2/16}$.
Extracting a subsequence again, we obtain that
$\chi'_0\in C^{1+\frac{\delta}{2}}(I_{r^2/16})$, and hence
$\chi_0\in C^{2+\frac{\delta}{2}}(I_{r^2/16})$.
Repeating this argument inductively yields that $\chi_0$ is smooth.
\end{proof}


All the results obtained so far can be reformulated in terms of the unique strong solution $\mathcal{Q}$ to the original obstacle problem~\eqref{obs:prev}.
\begin{thm}\label{thm:freeboundary_main}
    Let us denote $S_0(t)$ and $S_1(t)$ by $S_0(t)=e^{\chi_0(T-t)}$ and $S_1(t)=e^{\chi_1(T-t)}.$
    Then, the following properties are true:
    \begin{itemize}
        \item[(a)] $\mathcal{Q}\in W^{1,2}_{p,{\rm loc}}(\Omega_T)\cap C(\overline{\Omega_T})$, $1<p<\infty$.
        \item[(b)] $S_0\in C^\infty((0,T))$ and $S_1\in C^\infty((0,t_1))$.
    \end{itemize}
\end{thm}

\section{Optimal Strategies}\label{sec:strategy}

In this section, we briefly characterize the optimal strategies.
Since the verification argument, the duality theorem, and the recovery of the primal optimal controls follow the same line of reasoning as in~\cite[Section~5]{ZJ2025}, we state only the main results and omit the proofs.

By Theorem~\ref{thm:freeboundary_main}, the obstacle problem~\eqref{obs:prev} admits a unique strong solution $\mathcal{Q}\in W^{1,2}_{p,\mathrm{loc}}(\Omega_T)\cap C(\overline{\Omega_T})$, $1<p<\infty$, and the associated free boundaries $S_0$ and $S_1$ are smooth on their respective domains. Hence, by Theorem~\ref{thm:auxiliary_problem}, the system~\eqref{eq:auxiliary_problem} admits a unique strong solution pair $(P_0,P_1)$ satisfying the variational inequalities~\eqref{eq:vIs}. Adapting the verification argument in~\cite[Theorem~5.2]{ZJ2025} to the present setting with time-dependent switching costs, we obtain $P_j(0,y)=J(j,y)$ for all $(j,y)\in\{0,1\}\times(0,\infty)$.

For $y>0$, let $Y_t^y:=ye^{\beta t}\Upsilon_t$. The optimal switching strategy for the dual problem with initial state $(j,y)$ is determined by the first hitting times of the free boundaries and is given by
\begin{equation}\label{eq:optimal_job_state}
\widehat{\eta}_t^{(j,y)}
:=
j{\bf I}_{(0,\hat\tau_1]}(t)
+\sum_{n\ge1}\bigl(1-\widehat{\eta}^{(j,y)}_{\hat\tau_n^-}\bigr)
{\bf I}_{(\hat\tau_n,\hat\tau_{n+1}]}(t),
\end{equation}
where $\hat\tau_0=0$,
$\hat\tau_1=T\wedge\inf\{t>0:Y_t^y<S_0(t)\}$ if $j=0$ and $y>S_0(0)$,
$\hat\tau_1=0$ if $j=0$ and $y\le S_0(0)$,
$\hat\tau_1=0$ if $j=1$ and $y\ge S_1(0)$, and
$\hat\tau_1=t_1\wedge\inf\{t>0:Y_t^y>S_1(t)\}$ if $j=1$ and $y<S_1(0)$.
For $n\ge1$, we set
$\hat\tau_{n+1}=T\wedge\inf\{t>\hat\tau_n:Y_t^y<S_0(t)\}$ when $\widehat{\eta}^{(j,y)}_{\hat\tau_n^-}=0$, and
$\hat\tau_{n+1}=t_1\wedge\inf\{t>\hat\tau_n:Y_t^y>S_1(t)\}$ when $\widehat{\eta}^{(j,y)}_{\hat\tau_n^-}=1$.

We summarize the main result below.

\begin{thm}\label{thm:optimal_strategies}
For a given pair $(j,w)\in\{0,1\}\times(\iota(j),\infty)$, the value function satisfies the duality relation
\begin{equation}\label{eq:duality_main}
V(j,w)=\inf_{y>0}\bigl(J(j,y)+yw\bigr).
\end{equation}
Moreover, there exists a unique $y^*>0$ such that $w=-\partial_y J(j,y^*)$ and $V(j,w)=J(j,y^*)+y^*w$.

Let $Y_t^*:=y^*e^{\beta t}\Upsilon_t$ and $\eta_t^*:=\widehat{\eta}_t^{(j,y^*)}$. Then, on $[0,T)$, the optimal consumption, investment, and job-switching strategies are given by
\begin{equation}\label{eq:optimal_controls_main}
c_t^*=L_{\eta_t^*}^{\gam}(Y_t^*)^{-1/\gamma_1},
\qquad
\pi_t^*=\frac{\theta}{\sigma}Y_t^*\,\partial_{yy}P_{\eta_t^*}(t,Y_t^*),
\qquad
\xi_t^*=\z_{\eta_t^*}.
\end{equation}
Furthermore, the corresponding optimal wealth process satisfies $W_t^*=-\partial_y P_{\eta_t^*}(t,Y_t^*)$ for $0\le t<T$. After retirement, the optimal strategy coincides with the classical Merton rule: $c_t^*=\bar L^{\gam}(Y_t^*)^{-1/\gamma_1}$, $W_t^*=-J_R'(Y_t^*)$, and $\pi_t^*=\frac{\theta}{\sigma\gamma_1}W_t^*$ for $t\ge T$.
\end{thm}

The proof follows the same steps as in~\cite[Section~5]{ZJ2025}: one first identifies $(P_0,P_1)$ with the dual value functions via verification, then establishes~\eqref{eq:duality_main} by minimizing $J(j,y)+yw$ over $y>0$, and finally recovers the primal optimal controls from the first-order condition and the dual state process.


\bibliographystyle{siamplain}
\bibliography{library}

@Article{TangYong1993,
  author  = {S. Tang and J. Yong},
  title   = "{Finite horizon stochastic optimal switching and impulse controls with a viscosity solution approach}",
  journal = {Stochastics and Stochastics Reports},
  volume  = {45},
  number  = {3--4},
  pages   = {145--176},
  year    = {1993},
  month   = {},
}

@article{ZJ2025,
  author  = {Zhou Yang and Junkee Jeon},
  title   = {A Problem of Finite-Horizon Optimal Switching and Stochastic Control for Utility Maximisation},
  journal = {Finance and Stochastics},
  year    = {2026},
  volume  = {30},
  number  = {1},
  pages   = {59--118}, 
}

@Article{Yi,
	author = {F. Yi and Z. Yang},
	title = "{A variational inequality arising from European option pricing with transaction costs}",
	journal = {Science in China Series A: Mathematics},
	volume = {51},
	number = {5},
	pages = {935-954},
	year = {2008},
	month = {},
}

@Book{AS,
	author =	{M. Abramowitz and I. Stegun},
	title =	{Handbook of Mathematical Function},
	year =	 {1972},
	publisher= {Dover},
}

@Article{SKS,
	author = {G. Shim and J. Koo and Y. Shin},
	title = "{Reversible Job-Switching Opportunities and Portfolio
	Selection}",
	journal = {Applied Mathematics and Optimization},
	volume = {77},
	number = {},
	pages = {197-228},
	year = {2018},
	month = {},
}

@Article{M69,
	author = {R. Merton},
	title = "{Lifetime Portfolio Selection under Uncertainty: The continuous-time Case}",
	journal = {Review of Economics and Statistics},
	volume = {51},
	number = {3},
	pages = {247-257},
	year = {1969},
	month = {},
}

@Article{M71,
	author = {R. Merton},
	title = "{Optimum Consumption and Portfolio Rules in a Continuous-time Model}",
	journal = {Journal of Economic Theory},
	volume = {3},
	number = {1-2},
	pages = {373-413},
	year = {1971},
	month = {},
}

@Book{L,
	author =	 {D. G. Luenberger},
	title =	 "{Optimization by Vector Space Methods}",
	year =	 {1969},
	publisher= {John Wiley \& Sons, Inc.},
}

@Article{DY09,
	author = {M. Dai and F. Yi},
	title = "{Finite horizon optimal investment with transaction costs: a parabolic double obstacle problem}",
	journal = {Journal of Differential Equations},
	volume = {246},
	number = {4},
	pages = {1445-1469},
	year = {2009},
	month = {},
}

@Book{Lieberman,
	author =	{G. Lieberman},
	title =	{Second Order Parabolic Differential Equations},
	year =	 {1996},
	publisher= {World Scientific Singapore},
}

@article{Blanchet-et-al-2006,
	author = {A. Blanchet and J. Dolbeault and R. Monneau},
	title = "{On the continuity of the time derivative of the solution to the parabolic obstacle problem with variable coefficients}",
	journal = {Journal de Math\'{e}matiques Pures et Appliqu\'{e}es},
	volume = {85},
	number = {3},
	pages = {371-414},
	year = {2006},
	month = {11},
}

@article{Qi25,
title = {Job switching and bequest motives in an optimal consumption–investment model under inflation and mortality risks},
journal = {Economic Modelling},
volume = {153},
pages = {107307},
year = {2025},
issn = {0264-9993}, 
author = {Qi Li and Yong Hyun Shin and Ji-Hun Yoon}
}

@article{ShimJeon2025,
  author    = {Gyoocheol Shim and Junkee Jeon},
  title     = {Optimal Consumption and Investment with a Costly Reversible Job-Switching Option},
  journal   = {Mathematical Methods of Operations Research},
  year      = {2025},
  volume    = {101},
  number    = {3},
  pages     = {459--506},
  issn      = {1432-5217},
}

@article{JP23,
title = {Optimal job switching and retirement decision},
journal = {Applied Mathematics and Computation},
volume = {443},
pages = {127777},
year = {2023},
issn = {0096-3003}, 
author = {Junkee Jeon and Kyunghyun Park}
}

@Article{Dai2010,
author = {Dai, M. and Zhang, Q. and Zhu, Q. J.},
title = {Trend Following Trading under a Regime Switching Model},
journal = {SIAM Journal on Financial Mathematics},
volume = {1},
number = {1},
pages = {780-810},
year = {2010}
}

@article{Han2024,
author = {Han, Xiaoru and Yi, Fahuai and Zhang, Jianbo},
title = {A Reversible Investment Problem with Capacity and Demand in Finite Horizon: Free Boundary Analysis},
journal = {SIAM Journal on Control and Optimization},
volume = {62},
number = {2},
pages = {1207-1234},
year = {2024},
}

@article{CHEN2012928,
title = {Parabolic variational inequality with parameter and gradient constraints},
journal = {Journal of Mathematical Analysis and Applications},
volume = {385},
number = {2},
pages = {928-946},
year = {2012},
issn = {0022-247X},
author = {Xiaoshan Chen and Yingshan Chen and Fahuai Yi}
}

@article{an2025,
title = {Optimal portfolio and retirement decisions with costly job switching options},
journal = {Applied Mathematics and Computation},
volume = {491},
pages = {129215},
year = {2025},
issn = {0096-3003}, 
author = {Jongbong An and Junkee Jeon and Takwon Kim}
}

@article{Qi26,
title = {The effects of inflation risk on voluntary retirement and job switching by a martingale approach},
journal = {Mathematical Control and Related Fields},
volume = {16},
number = {0},
pages = {1-39},
year = {2026},
issn = {2156-8472},
author = {Qi Li and Yong Hyun Shin and Ji-Hun Yoon},
}

@article{SS14,
title = {An optimal job, consumption/leisure, and investment policy},
journal = {Operations Research Letters},
volume = {42},
number = {2},
pages = {145-149},
year = {2014},
issn = {0167-6377},
author = {Gyoocheol Shim and Yong Hyun Shin}
}

@article{Lee-et-al-2019,
	author = {H.S. Lee and G. Shim and Y.H. Shin},
	title = "{Borrowing Constraints, Effective Flexibility in Labor Supply, and Portfolio Selection}",
	journal = {Mathematics and Financial Economics},
	volume = {13},
	number = {2},
	pages = {173-208},
	year = {2019},
	month = {},
}

@Book{Ladyzhenskaya,
	author = {O. A. Ladyzhenskaya and V. A. Solonnikov and N. N. Ural'tseva},
	title = "{Linear and Quasi-linear Equations of Parabolic Type}",
	year = {1968},
	publisher = {American Mathematical Society},
	address = {Providence, RI}
}

@Book{Friedman,
  author    = {A. Friedman},
  title     = "{Partial Differential Equations of Parabolic Type}",
  publisher = {Robert E. Krieger Publishing Company, Inc.},
  year      = {1964},
  address   = {New York}
}

@article{YangYan2008,
author = {Yang, Zhou and Yan, Huiwen},
year = {2008},
month = {01},
pages = {},
title = {A comparison principle of a system of variational inequalities in unbounded set},
volume = {2008},
journal = {Journal of South China Normal University. Natural Science Edition}
}

@article{TorresLatorre2024,
  author = {Torres-Latorre, Clara},
  year = {2024},
  month = {08},
  pages = {73},
  title = {Parabolic Boundary Harnack Inequalities with Right-Hand Side},
  volume = {248},
  journal = {Archive for Rational Mechanics and Analysis},
  number = {5}, 
}

@article{Kukuljan2022,
  author  = {Teo Kukuljan},
  title   = {Higher order parabolic boundary Harnack inequality in $C^1$ and $C^{k, \alpha}$ domains},
  journal = {Discrete and Continuous Dynamical Systems},
  volume  = {42},
  number  = {6},
  pages   = {2667--2698},
  year    = {2022},
  issn    = {1078-0947}
}

@article{DA,
author = {De Angelis, Tiziano and Stabile, Gabriele},
title = {On Lipschitz Continuous Optimal Stopping Boundaries},
journal = {SIAM Journal on Control and Optimization},
volume = {57},
number = {1},
pages = {402-436},
year = {2019},
}

@article{HaJeonOk2025Chooser,
  title        = {The Obstacle Problem Arising from the American Chooser Option},
  author       = {Gugyum Ha and Junkee Jeon and Jihoon Ok},
  journal      = {arXiv preprint},
  volume       = {arXiv:2506.03623},
  year         = {2025},
  eprint       = {2506.03623},
  archivePrefix= {arXiv},
  primaryClass = {math.AP},
  note         = {Submitted June 4, 2025; Revised June 7, 2025},
}
\end{document}